\documentclass{amsart}

\usepackage{lipsum,amsmath,amsfonts,amssymb}
\usepackage{amsfonts}
\usepackage{graphicx}
\usepackage{epstopdf}
\usepackage{algorithmic}

\newcommand{\FK}[1]{{\color[rgb]{0,0,0}#1}}
\newcommand{\FFK}[1]{{\color[rgb]{0,0,0}#1}}
\newcommand{\MG}[1]{{\color[rgb]{0,0,0}#1}}
\newcommand{\MMG}[1]{{\color[rgb]{0,0,0}#1}}
\newcommand{\JS}[1]{{\color[rgb]{0,0,0}#1}}
\newcommand{\JJS}[1]{{\color[rgb]{0,0,0}#1}}
\newcommand{\JJJS}[1]{{\color[rgb]{0,0,0}#1}}
\newcommand{\R}{ \mathbb{R} }

\newcommand{\C}{ \mathbb{C} }
\newcommand{\cal}{ \mathcal }
\usepackage{tikz,enumitem}

\newtheorem{theorem}{Theorem}
\newtheorem{corollary}{Corollary}
\newtheorem{lemma}{Lemma}

\title{\MG{ParaOpt:} a parareal algorithm for optimality systems}

\author{Martin J. Gander}
\address{Section  of  Mathematics,  University  of  Geneva,  1211  Geneva  4,  Switzerland} 
\email{martin.gander@unige.ch}
\author{Felix Kwok}
\address{Department of Mathematics, Hong Kong Baptist University, Hong-Kong}
\email{felix\_kwok@hkbu.edu.hk}
\author{Julien Salomon}
\address{INRIA  Paris,  ANGE  Project-Team,  75589  Paris  Cedex  12,  France  and Sorbonne Universit\'e, CNRS, Laboratoire Jacques-Louis Lions, 75005 Paris, France} \email{julien.salomon@inria.fr}

\usepackage{amsopn}

\begin{document}

\maketitle

\begin{abstract}
  The time parallel solution of optimality systems arising in PDE
  \FFK{constrained} optimization could be achieved by simply applying any
  time parallel algorithm\MG{, such as Parareal,} to solve
  the forward and backward evolution problems arising in the
  optimization loop. We propose here a different strategy by devising
  directly a new time parallel algorithm\MG{, which we call ParaOpt,} for the
  coupled forward and backward nonlinear partial differential
  equations. \MG{ParaOpt} is inspired by the Parareal algorithm for
  evolution equations, and thus is automatically a two-level
  method. We provide a detailed convergence analysis for the case of
  linear parabolic PDE constraints. We illustrate the performance of
  \MG{ParaOpt} with numerical experiments both for linear and
  nonlinear optimality systems.
\end{abstract}



\section{Introduction}

Time parallel time integration has become an active research area over
the last decade; there is even an annual workshop now dedicated to
this topic called the P\MG{in}T (Parallel \FK{in} Time) workshop, which started
with the first such dedicated workshop at the USI in Lugano in June
2011.  The main reason for this interest is the advent of massively
parallel computers \cite{dongarra2011international} 
  with
so many computing cores that \MG{spatial parallelization of an
  evolution problem} saturates long before all cores have been
effectively used.  There are four classes of such algorithms: methods
based on multiple shooting leading to the parareal algorithm
\cite{Nievergelt:1964:PMI,Bellen:1989:PAI, kiehl1994parallel,
  lions2001resolution, gander2007analysis, gander:2008:nca,
  gander2008analysis, minion2010hybrid}, methods based on waveform
relaxation \cite{Lelarasmee:1982:WRM, gander1996overlapping,
  Gander:1998:STC, Giladi:1997:STD, Gander:2004:ABC, Gander:2007:OSW,
  Kwok:2014, Mandal:2014, al2014optimization, Gander:2016:ETNA},
methods based on multigrid
\cite{Hackbusch:1984:PMG,Lubich:1987:MGD,vandewalle1994space,
  horton1995space, emmett2012toward, Gander:2016:AON,
  falgout2014parallel, minion2015interweaving, dobrev2017two}, and
direct time parallel methods
\cite{Miranker:1967:PMF,sheen2003parallel, thomee2005high,
  maday2008parallelization, gander2013paraexp}; for a review of the
development of P\MG{in}T methods, see \cite{gander201550},\FFK{\cite{ongapplications}} and the references therein.

A natural area where \FK{this} type of parallelization could be used
effectively is \FK{in} PDE constrain\FK{ed} optimization \JS{on bounded time intervals}, when the constraint is a
time dependent PDE. \FK{In these problems, calculating the descent direction within the optimization
loop requires solving both a forward and a backward evolution problem, so one}
could directly apply time parallelization techniques to each of these
solves \FFK{\cite{gotschel2017parallel,gotschel2019efficient,gunther2018non,gunther2019non}}. 
\FFK{Parareal can also be useful in one-shot methods where the preconditioning operator
requires the solution of initial value problems, see e.g.~\cite{ulbrich2015preconditioners}.}
\JS{Another method\FK{, which}
has been proposed in~\cite{MST,RSSG} in the \FK{context} of quantum control,
\FK{consists of decomposing the time interval into sub-intervals and
defining intermediate states at sub-interval boundaries; this allows one to construct }
a set of independent optimization problems associated with \FK{each sub-interval in time.}
\FK{Each iteration of the method then requires the solution of these independent sub-problems in
parallel, followed by} a cheap update of the intermediate states.
}
\FK{In this paper, we} propose \FK{yet another} \JJJS{approach based on a
fundamental understanding of the parareal algorithm invented in
\cite{lions2001resolution} as a specific approximation of a multiple
shooting method \cite{gander2007analysis}. We construct} a new time\FK{-}parallel method \MG{called ParaOpt} for solving directly the coupled
forward and backward evolution problems arising in the optimal control
context. \FFK{
Our approach is related to the multiple shooting paradigm \cite{10.1145/355580.369128}, where the
time horizon is decomposed into non-overlapping sub-intervals, and we solve for the unknown interface
state and adjoint variables using an inexact Newton method so that the trajectories are continuous
across sub-intervals. Additionally, a parareal-like approximation is used to obtain a cheap approximate
Jacobian for the Newton solve. There are two potential benefits to our approach: firstly, it is known that
for some control problems, long time horizons lead to difficulties in convergence for the optimization
loop. Therefore, a multiple shooting approach allows us to deal with local subproblems on shorter time
horizons, where we obtain faster convergence. Such convergence enhancement has also been observed in \cite{BOCK19841603,MADAY2002387,ISI:000182496700006}, and also more recently in 
\cite{RSSG}. Secondly, if we use parareal to parallelize the forward and backward sweeps, then the
speedup ratio will be bounded above by $L/K$, where $L$ is the number of sub-intervals and $K$ is the number of parareal iterations required for convergence. For many problems, especially the non-diffusive ones like the Lotka-Volterra problem we consider in Section \ref{lotka_voltera_section}, this ratio does not go above 4--5; this limits the potential speedup that can be obtained from this classical approach. By decomposing the
control problem directly and conserving the globally coupled structure of the problem, we 
obtain higher speedup ratios, closer to ones that are achievable for two-level methods for elliptic problems.
}

Our paper is organized as follows: in Section \ref{Sec2}, we present
our PDE constrain\FK{ed} optimization model problem, and \MG{ParaOpt} for
its solution. In Section \ref{Sec3} we give a complete convergence
analysis of \MG{ParaOpt} for the case when the PDE constraint is
linear and of parabolic type.  We then illustrate the performance of
\MG{ParaOpt} by numerical experiments in Section \ref{Sec4}, both for
linear and nonlinear problems. We present our conclusions and an
outlook on future work in Section \ref{Sec5}.

\section{\MG{ParaOpt}: \FK{a two-grid method} for optimal control}\label{Sec2}

Consider the optimal control problem associated with the cost
functional
$$ J(c)= \frac 12\|y(T)-y_{target}\|^2+\frac\alpha 2 \int_0^T \|c(t)\|^2 dt,$$
where $\alpha>0$ is a fixed regularization parameter, $y_{target}$ is a target state, and {the evolution of} the state function $y{:[0,T]\to \R^n}$ is described by the {non-}linear equation
\begin{equation}\label{eq:dynamic}
\dot{y}(t)=f(y(t))+c(t),
\end{equation}
with initial condition $y(0)=y_{init}$, \FFK{where $c(t)$ is the control, which is}
{assumed to enter linearly in the forcing term.} The \FFK{first-order} optimality condition then reads
\begin{equation}\label{BVP}
  \dot{y}=f(y)-\frac{\lambda}{\alpha},\quad \dot{\lambda}=-(f'(y))^T\lambda,
\end{equation}
\MG{with the} final condition $\lambda(T)=y(T)-y_{target}$, see \cite{gander2014constrained} for a detailed derivation.

We now introduce a parallelization algorithm 
{for solving the coupled problem (\ref{eq:dynamic}--\ref{BVP}).}
The approach we propose follows the ideas of the
parareal algorithm, {combining a sequential coarse integration on $[0,T]$ and parallel fine
integration on subintervals.}

Consider a subdivision of
$[0,T]=\cup_{\ell=0}^{L-1}[T_{\ell},T_{\ell+1}]$ and two sets of
intermediate states $(\MG{Y}_\ell)_{\ell=0,\cdots,L}$ and
$(\MG{\Lambda}_\ell)_{\ell=1,\cdots,L}$ corresponding to approximations of
the state $y$ and the adjoint state $\lambda$ at times
$T_{0},\cdots,T_L$ and $T_{1},\cdots,T_L$ respectively.
%
%
%
We denote \FFK{by $P$ and $Q$} the nonlinear solution
operators for the boundary value problem (\ref{BVP}) on the subinterval
$[T_\ell,T_{\ell+1}]$ with initial condition $y(T_l)=\MG{Y}_l$ and final
condition $\lambda(T_{\ell+1})=\MG{\Lambda}_{\ell+1}$, \FFK{defined so that
$P$ propagates the state $y$ forward to $T_{\ell+1}$ and $Q$ propagates the
adjoint backward to $T_\ell$:}
\begin{equation}\label{PQdef}
  \left(\begin{array}{c}
  y(T_{\ell+1})\\
  \lambda(T_\ell)
  \end{array}\right)=
  \left(\begin{array}{c}
  P(\MG{Y}_\ell,\MG{\Lambda}_{\ell+1})\\
  Q(\MG{Y}_\ell,\MG{\Lambda}_{\ell+1})
  \end{array}\right).
\end{equation}
Using these solution operators, we can write the boundary value
problem as a system of subproblems, which have to satisfy the matching
conditions
\begin{equation}
  \begin{array}{rclrcl}
    \MG{Y}_0-y_{init}& = & 0, \\
    \MG{Y}_1-P(\MG{Y}_0,\MG{\Lambda}_1)&=& 0, \qquad& \MG{\Lambda}_1-Q(\MG{Y}_1,\MG{\Lambda}_2)&=&0,\\
    \MG{Y}_2-P(\MG{Y}_1,\MG{\Lambda}_2)&=& 0, & \MG{\Lambda}_2-Q(\MG{Y}_2,\MG{\Lambda}_3)&=&0,\\
                       &\vdots& &  &\vdots& \\
    \MG{Y}_L-P(\MG{Y}_{L-1},\MG{\Lambda}_L)&=& 0, & \MG{\Lambda}_L-\MG{Y}_L+y_{target}&=&0.
  \end{array}
\end{equation}
This nonlinear system of equations can be solved using Newton's
method. Collecting the unknowns in the vector
$(Y^T,\Lambda^T):=(\MG{Y}_0^T,\MG{Y}_1^T,\ldots,\MG{Y}_L^T,\MG{\Lambda}_1^T,\MG{\Lambda}_2^T,\ldots,\MG{\Lambda}_L^T)$,
we obtain the nonlinear system
$$
  {\cal F}\left(\begin{array}{c}
    Y\\
    \Lambda\\
  \end{array}\right):=\left(\begin{array}{c}
    \MG{Y}_0-y_{init}\\
    \MG{Y}_1-P(\MG{Y}_0,\MG{\Lambda}_1)\\
    \MG{Y}_2-P(\MG{Y}_1,\MG{\Lambda}_2)\\
   \vdots\\
   \MG{Y}_L-P(\MG{Y}_{L-1},\MG{\Lambda}_L)\\
   \MG{\Lambda}_1-Q(\MG{Y}_1,\MG{\Lambda}_2)\\
   \MG{\Lambda}_2-Q(\MG{Y}_2,\MG{\Lambda}_3)\\
   \vdots\\
   \MG{\Lambda}_L-\MG{Y}_L+y_{target}
  \end{array}\right)=0.
$$
Using Newton's method to solve this system gives the iteration
\begin{equation}\label{eq:Newton}
  {\cal F}'\left(\begin{array}{c}
    Y^{n}\\
    \Lambda^{n}\\
  \end{array}\right)\left(\begin{array}{c}
    Y^{n+1}-Y^n\\
    \Lambda^{n+1}-\Lambda^n\\
  \end{array}\right)=-{\cal F}\left(\begin{array}{c}
    Y^n\\
    \Lambda^n\\
  \end{array}\right),
\end{equation}
where the Jacobian matrix of ${\cal F}$ is given by
\begin{multline}\label{Jacobian}
  \arraycolsep0em
  {\cal F}'
  \left(\begin{array}{c}
    Y\\
    \Lambda\\
  \end{array}\right)=\\
  \footnotesize \arraycolsep1pt
  \left(\begin{array}{ccccc|ccccc}
     I &   &   &   &  &  & & & & \\
     -P_y(\MG{Y}_0,\MG{\Lambda}_1) & I &   &   &  & -P_\lambda(\MG{Y}_0,\MG{\Lambda}_1) & & & \\
          & \ddots & \ddots & & & & \ddots & & & \\
          &        & -P_y(\MG{Y}_{L-1},\MG{\Lambda}_L)\ & I & & & & &
     -P_\lambda(\MG{Y}_{L-1},\MG{\Lambda}_L)\\\hline
      & -Q_y(\MG{Y}_1,\MG{\Lambda}_2) & &  & & I\quad & -Q_\lambda(\MG{Y}_1,\MG{\Lambda}_2) & & &  \\
      & \hfill\ddots & & & & & \ddots\hfill &\ddots &  \\
      & &  -Q_y(\MG{Y}_{L-1},\MG{\Lambda}_L)& & &  & &I\quad & -Q_\lambda(\MG{Y}_{L-1},\MG{\Lambda}_L) \\
      & & & -I & &  & & & I \\
  \end{array}\right).
\end{multline}
Using the explicit expression for the Jacobian gives us the
componentwise linear system we have to solve at each Newton iteration:
\begin{equation}\label{NewtonComponentwise}
  \arraycolsep0em
  \begin{array}{rcl}
  \MG{Y}_0^{n+1} & = & y_{init}, \\
  \MG{Y}_1^{n+1} & = & -P(\MG{Y}_0^n,\MG{\Lambda}_1^n)+P_y(\MG{Y}_0^n,\MG{\Lambda}_1^n)(\MG{Y}_0^{n+1}-\MG{Y}_0^n)+
  P_\lambda(\MG{Y}_0^n,\MG{\Lambda}_1^n)(\MG{\Lambda}_1^{n+1}-\MG{\Lambda}_1^n),\\
  \MG{Y}_2^{n+1} & = & -P(\MG{Y}_1^n,\MG{\Lambda}_2^n)+P_y(\MG{Y}_1^n,\MG{\Lambda}_2^n)(\MG{Y}_1^{n+1}-\MG{Y}_1^n)+
  P_\lambda(\MG{Y}_1^n,\MG{\Lambda}_2^n)(\MG{\Lambda}_2^{n+1}-\MG{\Lambda}_2^n),\\
  & \vdots & \\
  \MG{Y}_L^{n+1} & = & -P(\MG{Y}_{L-1}^n,\MG{\Lambda}_L^n)+P_y(\MG{Y}_{L-1}^n,\MG{\Lambda}_L^n)(\MG{Y}_{L-1}^{n+1}-\MG{Y}_{L-1}^n)+
  P_\lambda(\MG{Y}_{L-1}^n,\MG{\Lambda}_L^n)(\MG{\Lambda}_L^{n+1}-\MG{\Lambda}_L^n),\\
  \MG{\Lambda}_1^{n+1} & = & Q(\MG{Y}_1^n,\MG{\Lambda}_2^n)+Q_\lambda(\MG{Y}_1^n,\MG{\Lambda}_2^n)(\MG{\Lambda}_2^{n+1}-\MG{\Lambda}_2^n)+Q_y(\MG{Y}_1^n,\MG{\Lambda}_2^n)(\MG{Y}_1^{n+1}-\MG{Y}_1^n),\\
  \MG{\Lambda}_2^{n+1} & = & Q(\MG{Y}_2^n,\MG{\Lambda}_3^n)+Q_\lambda(\MG{Y}_2^n,\MG{\Lambda}_3^n)(\MG{\Lambda}_3^{n+1}-\MG{\Lambda}_3^n)+Q_y(\MG{Y}_2^n,\MG{\Lambda}_3^n)(\MG{Y}_2^{n+1}-\MG{Y}_2^n),\\
  & \vdots & \\
  \MG{\Lambda}_{L-1}^{n+1} & = & Q(\MG{Y}_{L-1}^n,\MG{\Lambda}_L^n)+Q_\lambda(\MG{Y}_{L-1}^n,\MG{\Lambda}_L^n)(\MG{\Lambda}_L^{n+1}-\MG{\Lambda}_L^n)+Q_y(\MG{Y}_{L-1}^n,\MG{\Lambda}_L^n)(\MG{Y}_{L-1}^{n+1}-\MG{Y}_{L-1}^n),\\
  \MG{\Lambda}_L^{n+1}&=&\MG{Y}_L^{n+1}-y_{target}.
  \end{array}
\end{equation}
Note that this system is not triangular: the
$\MG{Y}_{\ell}^{n+1}$ are coupled to the $\MG{\Lambda}_{\ell}^{n+1}$ and vice
versa, which is clearly visible in the Jacobian in (\ref{Jacobian}).
\FK{This is in contrast to the initial value problem case, where the application
of multiple shooting leads to a block lower triangular system. }

The parareal approximation idea is to replace the derivative term by
a difference computed on a coarse grid in (\ref{NewtonComponentwise}),
\JS{i.e., to use the approximations}
\begin{equation}\label{stdParareal}
  \begin{array}{rcl}
  P_y(\MG{Y}_{\ell-1}^n,\MG{\Lambda}_\ell^n)(\MG{Y}_{\ell-1}^{n+1}-\MG{Y}_{\ell-1}^n)&
  \approx & P^G(\MG{Y}_{\ell-1}^{n+1},\MG{\Lambda}_\ell^n)-P^G(\MG{Y}_{\ell-1}^{n},\MG{\Lambda}_\ell^n),\\
  P_\lambda(\MG{Y}_{\ell-1}^n,\MG{\Lambda}_\ell^n)(\MG{\Lambda}_\ell^{n+1}-\MG{\Lambda}_\ell^n)&
  \approx & P^G(\MG{Y}_{\ell-1}^n,\MG{\Lambda}_\ell^{n+1})-P^G(\MG{Y}_{\ell-1}^{n},\MG{\Lambda}_\ell^n),\\
  Q_\lambda(\MG{Y}_{\ell-1}^n,\MG{\Lambda}_\ell^n)(\MG{\Lambda}_\ell^{n+1}-\MG{\Lambda}_\ell^n)&
  \approx & Q^G(\MG{Y}_{\ell-1}^n,\MG{\Lambda}_\ell^{n+1})-Q^G(\MG{Y}_{\ell-1}^{n},\MG{\Lambda}_\ell^n),\\
  Q_y(\MG{Y}_{\ell-1}^n,\MG{\Lambda}_\ell^n)(\MG{Y}_{\ell-1}^{n+1}-\MG{Y}_{\ell-1}^n)&
  \approx & Q^G(\MG{Y}_{\ell-1}^{n+1},\MG{\Lambda}_\ell^n)-Q^G(\MG{Y}_{\ell-1}^{n},\MG{\Lambda}_\ell^n),
  \end{array}
\end{equation}
{where $P^G$ and $Q^G$ are propagators obtained from a coarse discretization of the
subinterval problem \eqref{PQdef}, e.g., by using only one time step for the whole
subinterval.}
This is certainly cheaper than \FK{evaluating} the derivative on the fine
grid\FK{;} the remaining expensive fine grid operations
$P(\MG{Y}_{\ell-1}^n,\MG{\Lambda}_\ell^n)$ and $Q(\MG{Y}_{\ell-1}^n,\MG{\Lambda}_\ell^n)$
in (\ref{NewtonComponentwise}) can now all be performed in
parallel.
\FFK{However, since (\ref{NewtonComponentwise}) does not have a block triangular structure,
the resulting nonlinear system would need to be solved iteratively\JJJS{. Each }of these outer iterations is now
very expensive, since one must evaluate the propagators  $P^G({Y}_{\ell-1}^{n+1},{\Lambda}_\ell^n)$, etc., by solving a \emph{coupled nonlinear} local control problem. This is in contrast to initial value problems, where the additional cost of solving nonlinear local problems is justified, because the block lower triangular structure allows one to solve the outer problem by forward substitution, without the need to iterate.} 
%
%
\FFK{In order to reduce the cost of computing outer residuals, our idea is} not to use the parareal approximation \FFK{\eqref{stdParareal}}, but to
use the so\FK{-}called ``derivative parareal'' variant, where we approximate
the derivative by effectively computing it for a coarse problem, see
\cite{GanderHairer2014},
\begin{equation}\label{derivParareal}
  \begin{array}{rcl}
  P_y(\MG{Y}_{\ell-1}^n,\MG{\Lambda}_\ell^n)(\MG{Y}_{\ell-1}^{n+1}-\MG{Y}_{\ell-1}^n)&
  \approx & P_y^G(\MG{Y}_{\ell-1}^n,\MG{\Lambda}_\ell^n)(\MG{Y}_{\ell-1}^{n+1}-\MG{Y}_{\ell-1}^n),\\
  P_\lambda(\MG{Y}_{\ell-1}^n,\MG{\Lambda}_\ell^n)(\MG{\Lambda}_\ell^{n+1}-\MG{\Lambda}_\ell^n)&
  \approx & P_\lambda^G(\MG{Y}_{\ell-1}^n,\MG{\Lambda}_\ell^n)(\MG{\Lambda}_\ell^{n+1}-\MG{\Lambda}_\ell^n),\\
  Q_\lambda(\MG{Y}_{\ell-1}^n,\MG{\Lambda}_\ell^n)(\MG{\Lambda}_\ell^{n+1}-\MG{\Lambda}_\ell^n)&
  \approx & Q_\lambda^G(\MG{Y}_{\ell-1}^n,\MG{\Lambda}_\ell^n)(\MG{\Lambda}_\ell^{n+1}-\MG{\Lambda}_\ell^n),\\
  Q_y(\MG{Y}_{\ell-1}^n,\MG{\Lambda}_\ell^n)(\MG{Y}_{\ell-1}^{n+1}-\MG{Y}_{\ell-1}^n)&
  \approx & Q_y^G(\MG{Y}_{\ell-1}^n,\MG{\Lambda}_\ell^n)(\MG{Y}_{\ell-1}^{n+1}-\MG{Y}_{\ell-1}^n).
  \end{array}
\end{equation}

{The advantage of this approximation is that the computation of $P_y^G$, $P_\lambda^G$, etc.~only
involves linear problems. Indeed, \FFK{for a small perturbation $\delta y$ in $Y_{\ell-1}$,} the quantities $P_y^G(\MG{Y}_{\ell-1},\MG{\Lambda}_{\ell})\delta y$ and $Q_y^G(\FFK{{Y}_{\ell-1}},\FFK{{\Lambda}_{\ell}})\delta y$ can be computed by discretizing
and solving the coupled differential equations obtained by differentiating \eqref{BVP}. If $(y,\lambda)$ is the solution of (\ref{BVP}) with $y(T_{\ell-1}) = \MG{Y}_{\FFK{\ell-1}}$
and $\lambda(T_{\FFK{\ell}}) = \MG{\Lambda}_{\FFK{\ell}}$, then solving the \emph{linear} derivative system
\begin{align}\label{dBVP}
&\dot z = f'(y)z + \mu/\alpha, & &\dot \mu = -f'(y)^T\mu - H(y,\FK{z})^T\FK{\lambda},\\
&z(T_{\FFK{\ell-1}})= \delta y, & &\mu(T_{\FFK{\ell}}) = 0 \nonumber
\end{align}
\FK{on} a coarse \FK{time grid} leads to
$$z(T_{\FFK{\ell}}) = P_y^G(\MG{Y}_{\FFK{\ell-1}},\MG{\Lambda}_{\FFK{\ell}})\delta y, \qquad \mu(T_{\FFK{\ell-1}}) = Q_y^G(\MG{Y}_{\FFK{\ell-1}},\MG{\Lambda}_{\FFK{\ell}})\delta y, $$
where $H(y,z) = \lim_{r\to 0}\frac{1}{r}(f'(y+rz)-f'(y))$ is the Hessian of $f$ multiplied by $z$, and is thus linear in $z$.
\FFK{
Similarly, to compute $P_\lambda^G({Y}_{\FFK{\ell-1}},{\Lambda}_{\FFK{\ell}})\delta \lambda$
and $Q_\lambda^G({Y}_{\FFK{\ell-1}},{\Lambda}_{\FFK{\ell}})\delta \lambda$ for a perturbation $\delta \lambda$ in $\Lambda_\ell$, it suffices to solve the same ODE system as \eqref{dBVP}, except the end-point conditions must be replaced by $z(T_{\ell-1}) = 0$, $\mu(T_\ell) = \delta\lambda$.
}
Therefore, if GMRES is used to solve the Jacobian system \eqref{eq:Newton}, then each matrix-vector multiplication
requires only the solution of coarse, \FFK{\emph{linear}} subproblems in parallel, which is much cheaper than solving coupled nonlinear subproblems
in the standard parareal approximation \eqref{stdParareal}. 
}
\par {To summarize, \MG{our new ParaOpt} method consists of solving for $n=0,1,2,\ldots$ the system
\begin{equation}\label{OurMethod}
{\cal J}^G\left(\begin{array}{c}
    Y^{n}\\
    \Lambda^{n}\\
  \end{array}\right)\left(\begin{array}{c}
    Y^{n+1}-Y^n\\
    \Lambda^{n+1}-\Lambda^n\\
  \end{array}\right)=-{\cal F}\left(\begin{array}{c}
    Y^n\\
    \Lambda^n\\
  \end{array}\right),
\end{equation}
for $Y^{n+1}$ and $\Lambda^{n+1}$, where
\begin{multline}\label{approxJacobian}
  \arraycolsep0em
  {\cal J}^G
  \left(\begin{array}{c}
    Y\\
    \Lambda\\
  \end{array}\right)=\\
    \footnotesize \arraycolsep1pt
\left(\begin{array}{ccccc|ccccc}
     I &   &   &   &  &  & & & & \\
     -P_y^G(\MG{Y}_0,\MG{\Lambda}_1) & I &   &   &  & -P_\lambda^G(\MG{Y}_0,\MG{\Lambda}_1) & & & \\
          & \ddots & \ddots & & & & \ddots & & & \\
          &        & -P_y^G(\MG{Y}_{L-1},\MG{\Lambda}_L)\ & I & & & & &
     -P_\lambda^G(\MG{Y}_{L-1},\MG{\Lambda}_L)\\\hline
      & -Q_y^G(\MG{Y}_1,\MG{\Lambda}_2) & &  & & I\quad & -Q_\lambda^G(\MG{Y}_1,\MG{\Lambda}_2) & & &  \\
      & \hfill\ddots & & & & & \ddots\qquad\qquad\phantom{|} &\ddots &  \\
      & &  -Q_y^G(\MG{Y}_{L-1},\MG{\Lambda}_L)& & &  &\phantom{|}\qquad \quad I & & -Q_\lambda^G(\MG{Y}_{L-1},\MG{\Lambda}_L) \\
      & & & -I & &  & & & I \\
  \end{array}\right)
\end{multline}
is an approximation of the true Jacobian in \eqref{Jacobian}. If the
system \eqref{OurMethod} is solved using a matrix-free method, the
action of the sub-blocks $P_y^G$, $P_\lambda^G$, etc. can be obtained
by solving coarse linear subproblems of the type \eqref{dBVP}. 
\FFK{Note that the calculation of $ {\cal J}^G$ times a vector (without preconditioning) is 
embarrassingly parallel, since it only requires the solution of local subproblems of the type \eqref{dBVP},
with no additional coupling to other sub-intervals. Global communication is only required in two places:
within the Krylov method itself (e.g. when calculating inner products), and possibly within the preconditioner. The design of an effective preconditioner is an important and technical topic that will be the subject of a future paper. Of course, for problems with small state spaces (e.g. for ODE control problems), direct methods may also be used, once the coefficients of ${\cal J}^G$ are calculated by
solving  \eqref{dBVP} for suitable choices of $\delta y$ and $\delta \lambda$.}

\FFK{Regardless of how
\eqref{OurMethod} is solved, 
}
since we use an approximation of the Jacobian, the resulting \FFK{inexact Newton}
method will no longer converge quadratically, but only linearly; \FFK{this is true} even
in the case where the differential equation is linear.  In the next
section, we will analyze in detail the convergence of the method for
the case of a diffusive linear problem.  }

\section{Implicit Euler for the \JS{diffusive} linear case}\label{Sec3}

We now consider the method in a linear and discrete setting. More precisely, we focus on a control problem
\begin{equation}\label{eq:state_system}
 \dot{y}(t)=A y(t) + c(t),
\end{equation}
where $A$ is a real, symmetric matrix with {\bf negative} eigenvalues. The matrix $A$ can for example
be a finite difference discretization of a diffusion operator in space.
We \FK{will consider} 
a {\it discretize-then-optimize} strategy, so the analysis that 
follows is done in a discrete setting. 

\subsection{Discrete formulation}
{To fix ideas, we choose the implicit Euler\footnote{\MG{We use the term `implicit Euler' instead of `Backward Euler' because the method is applied forward and backward in time.}} method for the time discretization; other
discretizations will be studied in a future paper.}
Let $M\in \mathbb{N}$, and $\MG{\delta t=T/M}$. Then  the implicit Euler method gives\footnote{
\FK{If the ODE system contains mass matrices arising from a finite element discretization, e.g.,
$$ {\cal M}y_{n+1} = {\cal M}y_n + \delta t(Ay_{n+1} + {\cal M}c_{n+1}), $$
then one can analyze ParaOpt by introducing the change of variables $\bar{y}_n := {\cal M}^{1/2}y_n$, $\bar{c}_n :={\cal M}^{1/2}c_n$, so as to obtain
$$ \bar{y}_{n+1} = \bar{y}_n + \delta t(\bar{A}\bar{y}_{n+1} + \bar{c}_{n+1}), $$
with $\bar{A}:={\cal M}^{-1/2}A{\cal M}^{-1/2}$. Since $\bar{A}$ is symmetric positive definite whenever $A$ is, the analysis is identical to that for \eqref{BE_state}, even though one would never calculate $ {\cal M}^{1/2}$ and $\bar{A}$ in actual computations.}}
\begin{equation}\label{BE_state}
y_{n+1} =  y_n + \delta t (Ay_{n+1} + c_{n+1}),
\end{equation}
or, equivalently,
$$ y_{n+1} = (I - \delta t A)^{-1}(y_n + \delta t c_{n+1}). $$
We minimize the cost functional
$$ J_{\delta t}(c)=\frac 12\|y_M-y_{target}\|^2+\frac\alpha 2 \delta t\sum_{n=0}^{M-1} \|c_{n+1}\|^2.$$



For the sake of simplicity, we keep the notations $y$, $\lambda$ and $c$ for the discrete variables, that is $y=(y_n)_{n=0,\cdots,M}$, $\lambda=(\lambda_n)_{n=0,\cdots,M}$ and $c=(c_n)_{n=0,\cdots,M}$.
Introducing the Lagrangian \FFK{(see \cite{ito2008lagrange,gander2014constrained} and also \cite{glowinski1994exact,troltzsch2010optimal,pearson2012regularization} for details)}
$${\mathcal L}_{\delta t}(y,\lambda,c)=J_{\delta t}(c)- \sum_{n=0}^{M-1}\left\langle\lambda_{n+1},  y_{n+1} - (I - \delta tA)^{-1} (y_n +\delta t c_{n+1})\right\rangle, $$
the optimality systems reads:
\begin{align}
y_0&=y_{init},\label{eq:yinitDisc_sys}\\
y_{n+1}&=(I - \delta t A)^{-1}(y_n+ \delta t c_{n+1}),& \FFK{n=0,1,\ldots, M-1,}
\label{eq:stateDisc_sys}\\
\lambda_M&=y_M-y_{target},\label{eq:adjointFinDisc_sys}\\
\lambda_n&=(I - \delta t A)^{-1}\lambda_{n+1}, & \FFK{n=0,1,\ldots, M-1,}\label{eq:adjointDisc_sys}\\
\alpha c_{n+1} &= - (I - \delta t A)^{-1}\lambda_{n+1}, & \FFK{n=0,1,\ldots, M-1,}
\label{eq:controlDisc_sys}
\end{align}
where we used the fact that $A$ is symmetric. If $A = VDV^T$ is the eigenvalue decomposition of $A$,
then the transformation $y_n \mapsto V^Ty_n$, $\lambda_n \mapsto V^T\lambda_n$, $c_n\mapsto V^Tc_n$ allows us to diagonalize the equations \eqref{eq:yinitDisc_sys}--\eqref{eq:controlDisc_sys} and obtain
a family of {\bf decoupled} optimality systems of the form
\begin{align}
y_0&=y_{init},\label{eq:yinitDisc}\\
y_{n+1}&=(I - \sigma\delta t)^{-1}(y_n+ \delta t c_{n+1}),& \FFK{n=0,1,\ldots, M-1,}
\label{eq:stateDisc}\\
\lambda_M&=y_M-y_{target},\label{eq:adjointFinDisc}\\
\lambda_n&=(I - \sigma\delta t)^{-1}\lambda_{n+1},& \FFK{n=0,1,\ldots, M-1,}\label{eq:adjointDisc}\\
\alpha c_{n+1} &= - (I - \sigma\delta t)^{-1}\lambda_{n+1},& \FFK{n=0,1,\ldots, M-1,}
\label{eq:controlDisc}
\end{align}
where the $y_{\FFK{n}}$, $\lambda_{\FFK{n}}$ and $c_{\FFK{n}}$ are now scalars, and $\sigma < 0$ is an eigenvalue of $A$. This motivates us to study the scalar {\bf Dahlquist} problem
\begin{equation*}
 \dot{y}(t)=\sigma y(t) + c(t),
\end{equation*}
where $\sigma$ is a real, negative number. For the remainder of this section, we will study the
\MG{ParaOpt} algorithm applied to the
scalar variant \eqref{eq:yinitDisc}--\eqref{eq:controlDisc}, particularly its convergence properties
as a function of $\sigma$.

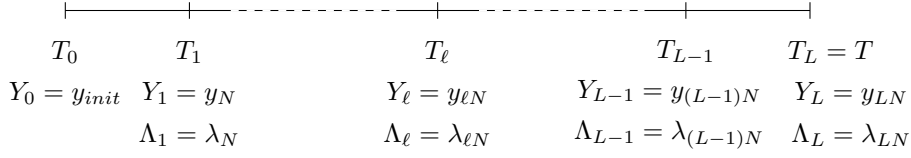
\begin{figure}
\begin{tikzpicture}
\begin{scope}[scale=1.1]

\draw (-5,0) -- (4,0);
\draw[white,dashed] (-3,0) -- (-1,0);
\draw[white,dashed] (0,0) -- (2,0);
\draw (-5,-.1) -- (-5,.1);  \draw (-3.5,-.1) -- (-3.5,.1);
\draw (4,-.1) -- (4,.1);    \draw (2.5,-.1) -- (2.5,.1);  
\draw (-.5,-.1) -- (-.5,.1);
\node at (-5,-.5) { $T_0$};\node at (-3.5,-.5) { $T_1$};
\node at ( 4.25,-.5) { $T_L=T$};\node at (2.5,-.5) { $T_{L-1}$};
\node at ( -.5,-.5) { $T_\ell$};

\node at (-5,-1) { $Y_0=y_{init}$};\node at (-3.5,-1) { $Y_1=y_N$};
\node at (4.5,-1) { $Y_{L}=y_{LN}$};
\node at ( -.5,-1) { $Y_\ell=y_{\ell N}$};
\node at (2.3,-1) { $Y_{L-1}=y_{(L-1)N}$};

\node at (-3.5,-1.5) { $\Lambda_1=\lambda_N$};
\node at ( 4.5,-1.5) { $\Lambda_L=\lambda_{LN}$};\node at (2.3,-1.5) { $\Lambda_{L-1}=\lambda_{(L-1)N}$};
\node at (-.5,-1.5) { $\Lambda_{\ell}=\lambda_{\ell N}$};
\end{scope}
\end{tikzpicture}
\caption{Notations associated with the parallelization setting.}\label{fig:not}
\end{figure}

\FK{Let us now write the linear
\MG{ParaOpt} algorithm for \eqref{eq:yinitDisc}--\eqref{eq:controlDisc}
in matrix form.} For the sake of simplicity, we assume that the subdivision is uniform, that is $T_\ell=\ell\Delta T $, where $N$ satisfies $\Delta T=N \delta t$ and $M=NL$, see Figure
\ref{fig:not}. We start by eliminating interior unknowns, i.e., ones that are not located
at the time points $T_0, T_1,\ldots T_L$. For $0\leq n_1 \leq n_2  \leq M$, \eqref{eq:stateDisc} and
\eqref{eq:controlDisc} together imply
\begin{align}\nonumber
y_{n_2}  &= (1-\sigma\delta t)^{n_1 - n_2} y_{n_1} -  \delta t
\sum_{j=0}^{n_2-n_1-1}(1-\sigma \delta t)^{n_1-n_2+j}c_{n_1+j+1}\\
&= (1-\sigma\delta t)^{n_1 - n_2} y_{n_1} -  \frac{\delta t}{\alpha}
\sum_{j=0}^{n_2-n_1-1}(1-\sigma \delta t)^{n_1-n_2+j-1}\lambda_{n_1+j+1}.\label{eq:yn_inter}
\end{align}
On the other hand, \eqref{eq:adjointDisc} implies
\begin{equation}\label{eq:adjoint_condensed}
\lambda_{n_1+j} = (1-\sigma \delta t)^{n_1-n_2+j} \lambda_{n_2}.
\end{equation}
Combining \eqref{eq:yn_inter} and \eqref{eq:adjoint_condensed} then leads to
\begin{equation}\label{eq:state_condensed}
y_{n_2} = (1-\sigma\delta t)^{n_1 - n_2} y_{n_1} -  \frac{\delta t}{\alpha}
\Bigl[\sum_{j=0}^{n_2-n_1-1}(1-\sigma \delta t)^{2(n_1-n_2+j)}\Bigr]\lambda_{n_2}.
\end{equation}
{
Setting $n_1 = (\ell-1)N$ and $n_2 = \ell N$, and using the notation $Y_\ell = y_{\ell N}$, $\Lambda_\ell
= \lambda_{\ell N}$ (see Figure \ref{fig:not}), we obtain from \eqref{eq:adjoint_condensed} and \eqref{eq:state_condensed} the equations
\begin{align*}
Y_0 &= y_{init}\\
-\beta_{\delta t} Y_{\ell-1} + Y_\ell + \frac{\gamma_{\delta t}}{\FFK{\alpha}} \Lambda_{\ell} &= 0, & 1\leq \ell \leq M,\\
\Lambda_{\ell-1} - \beta_{\delta t} \Lambda_\ell &= 0, & 0 \leq \ell \leq M-1,\\
Y_{\ell} + \Lambda_{\ell} &= y_{target},
\end{align*}
where 
\begin{eqnarray}
\beta_{\delta t} &:=& (1-\sigma\delta t)^{-\Delta T/\delta t}, \label{eq:beta_def}\\
\gamma_{\delta t}&:=& \delta t\sum_{j=0}^{N-1}(1-\sigma\delta t)^{2(j-N)}
= \frac{\beta_{\delta t}^2-1}{\sigma(2-\sigma \delta t)}.\label{eq:gamma_def}
\end{eqnarray}
}
In matrix form, this can be written as
$$\left( \begin{array}{cccc|ccccc}
1			&               	&	 			&   &         0      	& 			& 	& \\
-\beta_{\delta t}	&       \ddots  &	 			&   & \gamma_{\delta t}{/\alpha}	& \multicolumn{2}{c}{\ddots} 	& \\
              		&  	     \ddots	&	  \ddots		&   &  	      	&\multicolumn{2}{c}{\ddots}	      &  0\\
              		&  	      		&
							 -\beta_{\delta t} 	& 1 &           &  	      		&  	&  \gamma_{\delta t}{/\alpha}  \\
\hline
			&       	       	&   	  			&   & 1 		& -\beta_{\delta t}	& 	& \\
              		& 		 	& 	  			&   & 			& \ddots		& \ddots& \\
              		&  	      		& 	  			&   &  	      		& 			&  \ddots&  -\beta_{\delta t}\\
              		&  	      		&  	 		 	& -1 &               	&  	      		&  	&   1
\end{array}
\right)
\left( \begin{array}{c} Y_0 \\ 
\vdots \\ \vdots \\ Y_{L} \\\hline \Lambda_1 \\ \vdots \\ \vdots \\ \Lambda_L
\end{array}
\right)
=\left( \begin{array}{c} y_{init} \\ 0\\ \\ \\ \vdots \\ \\ \\ 0 \\ -y_{target}
\end{array}
\right),$$
or, in a more compact form,
\begin{equation}\label{discsys}
A_{\delta t}X=b.
\end{equation}
Note that this matrix has the same structure as the Jacobian matrix $\mathcal{F}$ in \eqref{Jacobian},
except that $Q_\lambda = 0$ for the linear case.
In order to solve \eqref{discsys} numerically, we consider a second time step $\Delta t$ such that $\delta t\leq \Delta t\leq \Delta T$. \FFK{In other words, for each sub-interval of length $\Delta T$, the derivatives of the propagators $P_y$, $Q_y$, $P_\lambda$, $Q_\lambda$ are approximated using a coarser
time discretization with time step $\Delta t \leq \Delta T$.}
\FK{The optimality system for this} 
 coarser time discretization has the form 
$$ A_{\Delta t}\hat{X} = b, $$
where $A_{\Delta t}$ has the same form as above, except that $\beta_{\delta t}$ and $\gamma_{\delta t}$
are replaced by $\beta_{\Delta t}$ and $\gamma_{\Delta t}$, i.e., the values obtained from the formulas \eqref{eq:beta_def} and \eqref{eq:gamma_def} when one replaces $\delta t$ by $\Delta t$. Then the
\FK{ParaOpt} 
algorithm (\ref{OurMethod}--\ref{approxJacobian}) for the linear Dahlquist problem can be written as
$$A_{\Delta t}(X^{k+1}-X^{k})=\FFK{-\mathcal{F}(X^k)} = -(A_{\delta t}X^k-b),$$
or, equivalently
\begin{equation}\label{eq:scheme}
X^{k+1} = \left(I-A_{\Delta t}^{-1}A_{\delta t}\right)X^{k}+A_{\Delta t}^{-1}b.
\end{equation}
Note that using this iteration, only a coarse matrix needs to be inverted.

\subsection{Eigenvalue problem}
In order to study the convergence of the iteration \eqref{eq:scheme}, we \FK{study the eigenvalues}
of the matrix $I-A_{\Delta t}^{-1}A_{\delta t}$, \FK{which} are given
by the generalized eigenvalue problem
\begin{equation}\label{eq:eig_prob}
(A_{\Delta t}-A_{\delta t})x=\mu A_{\Delta t}x,
\end{equation}
with $x=(v_0,v_1,\cdots,v_L,w_1,\cdots,w_L)^T$ being the eigenvector associated with the eigenvalue $\mu$.
{Since $A_{\Delta t}-A_{\delta t}$ has two zero rows, the eigenvalue $\mu=0$ must have multiplicity at least two.
Now let $\mu\neq 0$ be a non-zero eigenvalue. (If no such eigenvalue exists, then the preconditioning matrix is nilpotent and the iteration converges  in a finite number of steps.) Writing \eqref{eq:eig_prob} componentwise yields
}
\begin{eqnarray}
v_0&=& 0 \label{Compo:eq1}\\
\mu(v_\ell-\beta v_{\ell-1} +\gamma w_\ell{/\alpha}) &=& -\delta\beta v_{\ell-1}
+\delta\gamma w_{\ell}{/\alpha}
\label{Compo:eq2} \\
\mu(w_\ell-\beta w_{\ell+1})&=&-\delta\beta w_{\ell+1}\label{Compo:eq3}\\
\mu\left(w_L -v_L\right)&=&0,\label{Compo:eq4}
\end{eqnarray}
where we have introduced the simplified notation
\begin{equation}\label{eq:simplified_notations}
\beta  = \beta_{\Delta t},\ \gamma = \gamma_{\Delta t},\ \delta\beta=\beta_{\Delta t}-\beta_{\delta t},\
\delta\gamma=\gamma_{\Delta t}-\gamma_{\delta t}.
\end{equation}
{The recurrences \eqref{Compo:eq2} and~\eqref{Compo:eq3} are of the form}
\begin{equation}\label{vw_recurrence}
v_\ell = a v_{\ell-1} + bw_\ell, \qquad
w_\ell= aw_{\ell+1},
\end{equation}
where
$$
a = \beta  - \mu^{-1}\delta\beta, \qquad
b = \frac{-\gamma + \mu^{-1}\delta\gamma}{\alpha}.
$$
Solving the recurrence \eqref{vw_recurrence} in $v$ together with the initial condition \eqref{Compo:eq1}
leads to
\begin{equation}\label{eq:etape1}
v_L=\sum_{\ell=1}^{L}a^{L-\ell}bw_\ell,
\end{equation}
whereas the recurrence \eqref{vw_recurrence} in $w$ simply gives
\begin{equation}\label{eq:etape2}
w_\ell=a^{L-\ell}w_L.
\end{equation}
Combining~\eqref{eq:etape1} and~\eqref{eq:etape2}, we obtain $$v_L=\left(\sum_{\ell=1}^{L}a^{2(L-\ell)}b\right)w_L,$$ so that \eqref{Compo:eq4} gives rise to $P(\mu)w_L=0$, with
\begin{equation}\label{poly_def}
P(\mu)= \alpha\mu^{2L-1} + (\mu\gamma - \delta\gamma)\sum_{\ell=0}^{L-1}\mu^{2(L-\ell-1)}(\mu\beta-\delta\beta)^{2\ell}.
\end{equation}
Since we seek a non-trivial solution, we can assume $w_L \neq 0$. {Therefore, the eigenvalues of $I-A_{\Delta t}^{-1}A_{\delta t}$
consist of the number zero (with multiplicity two), together with the $2L-1$ roots of $P(\mu)$, which are all non-zero.
In the next subsection, we will give a precise characterization of the roots of $P(\mu)$, which depend on $\alpha$, as well as on $\sigma$ via the parameters $\beta$, $\delta \beta$, $\gamma$ and $\delta \gamma$.
}

\subsection{Characterization of eigenvalues}
In the next two results, we describe the location of the roots of $P(\mu)$ from the last section, or equivalently, the non-zero eigenvalues of the \FK{iteration} 
matrix $I - A_{\Delta t}^{-1}A_{\delta t}$. We first establish the sign of a few parameters in the case $\sigma < 0$, which is true for diffusive problems.

\begin{lemma}\label{lem:sign_of_parameters}
Let $\sigma < 0$. Then we have $0 < \beta < 1$, $0 < \delta \beta < \beta$, $\gamma > 0$ and
$\delta\gamma < 0$.
\end{lemma}
\begin{proof}
By the definitions \eqref{eq:beta_def} and \eqref{eq:simplified_notations}, we see that
$$\beta = \beta_{\Delta t} = (1-\sigma\Delta t)^{-\Delta T/\Delta t},$$
which is between $0$ and $1$, since $1-\sigma\Delta t > 1$ for $\sigma < 0$. Moreover, $\beta_{\Delta t}$
is an increasing function of $\Delta t$ by direct calculation, so that
$$ \delta \beta = \beta_{\Delta t} - \beta_{\delta t} > 0, $$
which shows that $0 < \delta \beta < \beta$. Next, we have by definition
$$ \gamma = \frac{\beta^2-1}{\sigma (2-\sigma\Delta t)}. $$
Since $\beta < 1$ and $\sigma < 0$, both the numerator and the denominator are negative, so $\gamma > 0$.
Finally, we have
$$
\delta\gamma = \frac{1}{|\sigma|}\left(\frac{1-\beta_{\Delta t}^2}{2+|\sigma|\Delta t} -
\frac{1-\beta_{\delta t}^2}{2+|\sigma|\delta t}\right) < 0, $$
since $1-\beta_{\Delta t}^2 < 1-\beta_{\delta t}^2$ and $2+|\sigma|\Delta t > 2+|\sigma|\delta t$, so the first quotient inside the parentheses is necessarily smaller than the second quotient.
\end{proof}

We are now ready to prove a first estimate for the eigenvalues of the matrix $I - A_{\Delta t}^{-1}A_{\delta t}$.
\begin{theorem}\label{thm:spectral}
Let $P$ be the polynomial defined in \eqref{poly_def}. For $\sigma < 0$, the roots of $P$ are contained in the set $D_\sigma\cup \{\mu^*\}$, where
\begin{equation}\label{Dsigma_def}
D_\sigma = \{\mu\in \C: |\mu-\mu_0| < \delta\beta/(1-\beta^2)\},
\end{equation}
where $\mu_0 = -\beta\delta\beta/(1-\beta^2)$, and $\mu^*  < 0$  is a real negative number.
\end{theorem}
\begin{proof}
Since zero is not a root of $P(\mu)$, we can divide $P(\mu)$ by $\mu^{2L-1}$ and see that $P(\mu)$ has
 the same roots as the function
 $$\hat{P}(\mu) = \alpha + (\gamma - \mu^{-1}\delta\gamma)\sum_{\ell=0}^{L-1}(\beta - \mu^{-1}\delta\beta)^{2\ell}.$$
Recall the change of variables
 $$a = \beta - \mu^{-1}\delta \beta\quad \iff \quad\mu = \frac{\delta \beta}{\beta -a}\;;$$
\FFK{substituting $a$ into $\hat{P}(\mu)$ and multiplying the result by $\delta \beta/|\delta\gamma|$
shows that} $P(\mu) = 0$ is equivalent to
$$Q(a) := \frac{\alpha\delta \beta}{|\delta\gamma|} + (C - a)\sum_{\ell=0}^{L-1}a^{2\ell} = 0,$$
with
\begin{equation}\label{eq:C_def}
C := \beta + \gamma\delta\beta/|\delta\gamma|> 0.
\end{equation}
We will now show that $Q(a)$ has at
most one root inside the unit disc $|a| \leq 1$; since the transformation from $\mu$ to $a$ maps
circles to circles, this would be equivalent to proving that $P(\mu)$ has at most one root outside the
disc $D_\sigma$. We now use the argument principle from complex analysis, which states that the difference between the
number of zeros and poles of $Q$ inside a closed contour $\mathcal{C}$ is equal to the winding number
of the contour $Q(\mathcal{C})$ around the origin. Since $Q$ is a polynomial and has no poles, this would allow us to count the number of zeros of $Q$ inside the unit disc. Therefore, we consider the winding number of the
contour $\Gamma = \{f(e^{i\theta}):0\leq \theta \leq 2\pi\}$ with
$$f(a) = (C - a)\sum_{\ell=0}^{L-1}a^{2\ell}$$
 around the point $-\alpha\delta\beta/|\delta\gamma|$, which is a real negative number. If we can show
 that $\Gamma$ intersects the negative real axis at at most one point, then it follows that the winding number around any negative real number cannot be greater than 1.

We now concentrate on finding values of $\theta$ such that
 $\arg(f(e^{i\theta}))=\pi \pmod{2\pi}$. Since $f(\overline{a}) = \overline{f(a)}$, it suffices to consider
 the range $0 \leq \theta \leq \pi$, and the other half of the range will follow by conjugation.
Since $f$ is a product, we deduce that
$$ \arg(f(e^{i\theta})) = \arg(C - e^{i\theta}) + \arg\left(1+e^{2i\theta}+\cdots + e^{2(L-1)i\theta}\right).$$
\begin{figure}[t]
\centering
\includegraphics[width=0.7\textwidth]{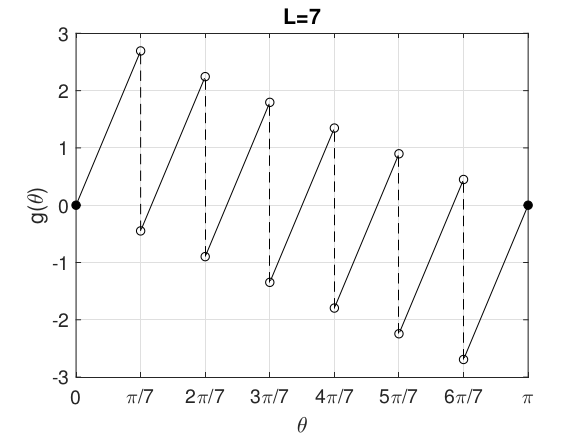}\\
\caption{\FFK{Plot of} $\FFK{g(\theta) = \arg}\left(\sum_{\ell=0}^{L-1} e^{2i\ell\theta}\right)$ for $L=7$. \label{fig:arg}}
\end{figure}
We consider the two terms on the right separately.
\begin{itemize}
\item For the first term, we have for all $0 < \theta < \pi$
$$ \theta-\pi < \arg(-e^{i\theta}) < \arg(C-e^{i\theta}) < 0, $$
since $C$ is real and positive. For $\theta = \pi$, we obviously have $\arg(C-e^{i\theta}) = 0$,
whereas for $\theta = 0$, we have $\arg(C-e^{i\theta}) = -\pi$ if $C < 1$, and $\arg(C-e^{i\theta}) = 0$
otherwise.
\item For the second term, observe that for $0 < \theta < \pi$, we have
$$ 1+e^{2i\theta}+\cdots + e^{2(L-1)i\theta} = \frac{1-e^{2iL\theta}}{1-e^{2i\theta}} = e^{(L-1)i\theta} \cdot \frac{\sin(L\theta)}{\sin(\theta)}. $$
Therefore, the second term is piecewise linear with slope $L-1$, with a jump of size $\pi$ whenever
$\sin(L\theta)$ changes sign, i.e., at $\theta = k\pi/L$, $k=1,\ldots, L-1$. Put within the range $(-\pi,\pi)$,
we can write
$$  \arg\left(\frac{1-e^{2iL\theta}}{1-e^{2i\theta}}\right) = (L-1)\theta - \left\lfloor \frac{L\theta}{\pi}\right\rfloor\pi =: g(\theta), \qquad 0 < \theta < \pi. $$
We also have $g(0) = g(\pi) = 0$ by direct calculation. The function $g$ satisfies the property
$ -\theta \leq g(\theta) \leq \pi - \theta$, see Figure \ref{fig:arg}.
 \end{itemize}
 From the above, we deduce that $\arg(f(e^{i\theta})) < \pi$ for all $0 \leq \theta \leq \pi$. Moreover,
 $$ \arg(f(e^{i\theta})) = \begin{cases}
 0, & \text{if $\theta = 0$ and $ C > 1$,}\\
 -\pi, & \text{if $\theta = 0$ and $ C < 1$,}\\
 \arg(C-e^{i\theta}) + g(\theta) > -\pi, & \text{if $0 < \theta < \pi$,}\\
 0, & \text{if $\theta = \pi$}.
 \end{cases}$$
\label{astar_argument} Thus, the winding number around the point $-\alpha\delta\beta/|\delta\gamma|$ cannot exceed one,
so at most one of the roots of $Q$ can lie inside the unit disc. If there is indeed such a root $a^*$, it must be real,
since the conjugate of any root of $Q$ is also a root. Moreover, \MG{it} must satisfy $a^* > C$, since $Q(a) > 0$
for any $a\leq C$. This implies
\begin{equation}\label{eq:astar_bound}
\beta - a^* < \beta - C = -\frac{\gamma\delta \beta}{|\delta\gamma|} < 0,
\end{equation}
so the corresponding $\mu^* = \delta\beta/(\beta-a^*)$ must also be negative.
\end{proof}

We have seen that the existence of $\mu^*$ depends on whether the constant $C$ is larger than 1.
The following lemma shows that we indeed have $C < 1$.

\begin{lemma} Let $\sigma < 0$. Then the constant $C = \beta + \gamma\delta\beta/|\delta\gamma|$,
defined in \eqref{eq:C_def}, satisfies $C < 1$.\label{constant_C}
\end{lemma}
\begin{proof}
We first transform the relation $C<1$ into a sequence of equivalent inequalities. Starting with the definition of $C$,
we have
\begin{align*}
C = \beta_{\Delta t} + \frac{\gamma_{\Delta t}(\beta_{\Delta t}-\beta_{\delta t})}{\gamma_{\delta t}
-\gamma_{\Delta t}} < 1
&\iff \beta_{\Delta t}(\gamma_{\delta t}-\gamma_{\Delta t}) + \gamma_{\Delta t}(\beta_{\Delta t}-\beta_{\delta t}) < \gamma_{\delta t}-\gamma_{\Delta t}\\
&\iff \gamma_{\Delta t}(1-\beta_{\delta t}) < \gamma_{\delta t}(1-\beta_{\Delta t})\\
&\iff \frac{(1-\beta_{\Delta t}^2)(1-\beta_{\delta t})}{|\sigma|(2+|\sigma|\Delta t)} <
\frac{(1-\beta_{\delta t}^2)(1-\beta_{\Delta t})}{|\sigma|(2+|\sigma|\delta t)}\\
&\iff \frac{1+\beta_{\Delta t}}{2+|\sigma|\Delta t} < \frac{1+\beta_{\delta t}}{2+|\sigma|\delta t},
\end{align*}
\FFK{where the last equivalence is obtained by multiplying both sides of the penultimate inequality by $|\sigma|$
and then dividing it by $(1-\beta_{\Delta t})(1-\beta_{\delta t})$.}
By the definition of $\beta_{\Delta t}$ and $\beta_{\delta t}$, the last inequality can be written as
$f(|\sigma|\Delta t) < f(|\sigma|\delta t)$, where
$$ f(x) := \frac{1+(1+x)^{-k/x}}{2+x} $$
with $k=|\sigma|\Delta T > 0$. Therefore, it suffices to show that $f(x)$ is decreasing for $0 \;\FFK{<}\; x \leq k$.
In other words, we need to show that
$$ f'(x) = \frac{(1+x)^{-k/x}}{2+x}\left[\frac{k\ln(1+x)}{x^2}  - \frac{k}{x(1+x)}\right] -
\frac{1+(1+x)^{-k/x}}{(2+x)^2} < 0. $$
%
%
This is equivalent to showing
\begin{equation}\label{exp_eqn}
(2+x)\left[\frac{k\ln(1+x)}{x^2}  - \frac{k}{x(1+x)}\right] -1 < (1+x)^{k/x}.
\end{equation}
Using the fact that $\ln(1+x) \leq x$, we see that the left hand side is bounded above by
\begin{align*}
(2+x)\left[\frac{k\ln(1+x)}{x^2}  - \frac{k}{x(1+x)}\right] -1
&\leq  (2+x)\left[\frac{kx}{x^2}  - \frac{k}{x(1+x)}\right] -1 \\
&= k\left(\frac{2+x}{1+x}\right) - 1.
\end{align*}
But for every $k>0$ and $0< x < k$ we have
\begin{equation}\label{technical_estimate}
(1+x)^{k/x} > k\left(\frac{2+x}{1+x}\right)-1,
\end{equation}
see proof in the appendix. Therefore, \eqref{exp_eqn} is satisfied by all $k>0$ and $0 < x < k$, so $f$ is in fact
decreasing. It follows that $C<1$, as required.
\end{proof}

\begin{theorem}
Let $\sigma < 0$ be fixed, and let
\begin{equation}\label{L0_def}
L_0 := \frac{C-\beta}{\gamma(1-C)}.
\end{equation}
Then the spectrum of $I-A_{\Delta t}^{-1}A_{\delta t}$ has an eigenvalue $\mu^*$ outside the disc $D_{\sigma}$ defined in \eqref{Dsigma_def} if and only if the
number of subintervals $L$ satisfies $L > \alpha L_0$\MG{, where $\alpha$ is
  the regularization parameter}.
\end{theorem}

\begin{proof}
The isolated eigenvalue exists if and only if the winding number of $Q(e^{i\theta})$ about the origin is non-zero. Since $Q(e^{i\theta})$ only
intersects the negative real axis at most once, we see that the winding number is non-zero when $Q(-1) < 0$, i.e., when
$$ \frac{\alpha\delta\beta}{|\delta\gamma|} + (C-1)L < 0.$$
Using the definition of $C$, this leads to
$$ \frac{\alpha(C-\beta)}{\gamma}+(C-1)L < 0 \iff L > \frac{\alpha(C-\beta)}{\gamma(1-C)}, $$
hence the result.
\end{proof}

\subsection{Spectral radius estimates}
The next \MG{theorem} now gives a more precise estimate on the isolated eigenvalue $\mu^*$.
\begin{theorem}\label{thm:isolated}
Suppose that the number of intervals $L$ satisfies $L > \alpha L_0$, with $L_0$ defined in
\eqref{L0_def}. Then the real negative eigenvalue $\mu^*$ outside the disc $D_\sigma$
is bounded below by
$$
\mu^* > -\frac{|\delta\gamma|+\alpha\delta\beta(1+\beta)}{\gamma+\alpha(1-\beta^2)}.
$$
\end{theorem}
\begin{proof}
Suppose $a^* = \beta - \delta\beta/\mu^*$ is a real root of $Q(a)$ inside the unit disc. We have
seen at the end of the proof of Theorem \ref{thm:spectral} (page
\pageref{astar_argument}, \FFK{immediately before Equation \eqref{eq:astar_bound}) that $a^* > C$;
moreover, since $a^*$ is assumed to be inside the unit circle, we must have $a^* < 1$. 
Therefore, $a^*$  satisfies  $C < a^* < 1$.} This implies
$$ \frac{\alpha\delta\beta}{|\delta\gamma|} + \frac{C-a^*}{1-(a^*)^2} = \frac{(C-a^*)(a^*)^{2L}}{1-(a^*)^2} < 0. $$
Therefore, $a^*$ satisfies
$$ (1-(a^*)^2)\alpha\delta\beta + |\delta\gamma|(C-a^*) < 0\MG{,} $$
which means
\begin{align*}
 a^* >& \frac{-|\delta\gamma|+\sqrt{|\delta\gamma|^2 + 4\alpha\delta\beta(\alpha\delta\beta+C|\delta\gamma|)}}{2\alpha\delta\beta}\\
=& \frac{-|\delta\gamma|+\sqrt{(|\delta\gamma|^2 + 2\alpha\delta\beta)^2-4(1-C)\alpha\delta\beta|\delta\gamma|}}{2\alpha\delta\beta}.
\end{align*}
Therefore,
\begin{align*}
\mu^*&=\frac{\delta\beta}{\beta-a^*}\\
&> \frac{2\alpha\delta\beta^2}{(2\alpha\beta\delta\beta+|\delta\gamma|)-\sqrt{(|\delta\gamma| + 2\alpha\delta\beta)^2-4(1-C)\alpha\delta\beta|\delta\gamma|}}\\
&= \frac{2\alpha\delta\beta^2\left[(2\alpha\beta\delta\beta+|\delta\gamma|)+\sqrt{(|\delta\gamma| + 2\alpha\delta\beta)^2-4(1-C)\alpha\delta\beta|\delta\gamma|}\right]}
{(2\alpha\beta\delta\beta+|\delta\gamma|)^2-(|\delta\gamma|+2\alpha\delta\beta)^2+4(1-C)\alpha\delta\beta|\delta\gamma|}\\
&= \frac{\delta\beta\left[(2\alpha\beta\delta\beta+|\delta\gamma|)+\sqrt{(|\delta\gamma| + 2\alpha\delta\beta)^2-4(1-C)\alpha\delta\beta|\delta\gamma|}\right]}
{2(\beta-C)|\delta\gamma|+2\alpha\delta\beta(\beta^2-1)}\\
&= -\frac{(2\alpha\beta\delta\beta+|\delta\gamma|)+\sqrt{(|\delta\gamma| + 2\alpha\delta\beta)^2-4(1-C)\alpha\delta\beta|\delta\gamma|}}
{2\gamma+2\alpha(1-\beta^2)}\\
&> -\frac{|\delta\gamma|+\alpha\delta\beta(1+\beta)}{\gamma+\alpha(1-\beta^2)},
\end{align*}
where the last inequality is obtained by dropping the term containing $(1-C)$ inside the square root, which
makes the square root larger since $C < 1$.
\end{proof}

To illustrate the above theorems, we show in Figures \ref{fig3.3} and \ref{fig3.4} the spectrum of the \FK{iteration} 
matrix $I-A_{\Delta t}^{-1}A_{\delta t}$ for different values of
$\sigma$ and for $\alpha = 1$ and 1000. Here, the time interval $[0,T]$ is subdivided into $L=30$ subintervals, and each subinterval contains 50 coarse time steps and 5000 fine time steps.
Table~\ref{table3.1} shows the values of the relevant parameters. For $\alpha = 1$, we see that
there is always one isolated eigenvalue on the negative real axis, since $L > L_0$ in all cases, and
its location is predicted rather accurately by the formula~\eqref{mu_bound}. The rest of the eigenvalues
\FK{all lie within the disc $D_\sigma$ defined in}~\eqref{Dsigma_def}. For $\alpha = 1000$,
the bounding disc is identical to the previous case; however, since
we have $L < \alpha L_0$ for all cases except for $\sigma = -16$, we observe no
eigenvalue outside the disc, except for \MG{the} very last case. In that very last case, we have $|\delta\gamma| = 0.0107$, so \eqref{mu_bound} gives the lower bound $\mu^* > -1.07\times 10^{-5}$, which again
is quite accurate when compared with the bottom right panel of Figure \ref{fig3.4}.

\begin{table}
\centering
\caption{Parameter values for $T=100$, $L=30$, $\Delta T/\Delta t =50$, $\Delta t/\delta t=100$.
\label{table3.1}}
{\small \tabcolsep4pt
\begin{tabular}{|c||c|c|c|c|c|c|}
\hline
$\sigma$ & $\beta$ & $\gamma$ & $C$ & $L_0$ & Radius of $D_\sigma$ & $\mu^*$ bound ($\alpha=1$) \\
\hline
$-1/8$ & 0.6604 & 2.2462 & 0.8268 & 0.4280 & $2.00\times 10^{-3}$ & $\FFK{-}6.08\times 10^{-3}$\\
$-1/4$ & 0.4376 & 1.6037 & 0.6960 & 0.5300 & $3.67\times 10^{-3}$ & $\FFK{-}9.34\times 10^{-3}$\\
$-1/2$ & 0.1941 & 0.9466 & 0.4713 & 0.5539 & $5.35\times 10^{-3}$ & $\FFK{-}1.24\times 10^{-2}$\\
$-1$ & 0.0397 & 0.4831 & 0.1588 & 0.2930 & $3.97\times 10^{-3}$ & $\FFK{-}1.36\times 10^{-2}$\\
$-2$ & 0.0019 & 0.2344 & 0.0116 & 0.0417 & $6.36\times 10^{-4}$ & $\FFK{-}1.30\times 10^{-2}$\\
$-16$ & $1.72\times 10^{-16}$ & 0.0204 & $5\times 10^{-16}$ & $1.61\times 10^{-14}$ & $1.72\times 10^{-16}$ & $\FFK{-}1.05\times 10^{-2}$ \\
\hline
\end{tabular}
}
\end{table}
\begin{figure}
\centering
\rotatebox{90}{\hskip 55pt $\mathrm{Im}(\mu)$}
\includegraphics[width=0.47\textwidth]{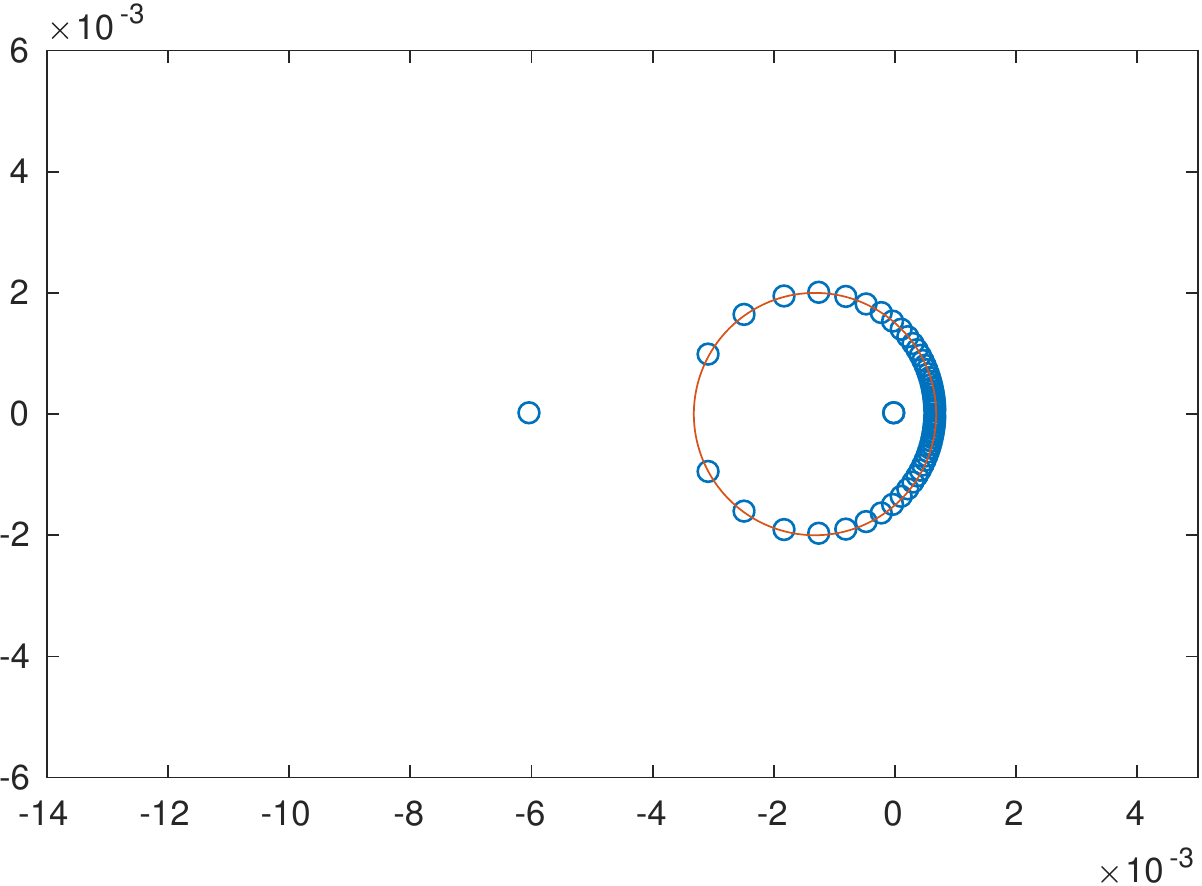}
\includegraphics[width=0.47\textwidth]{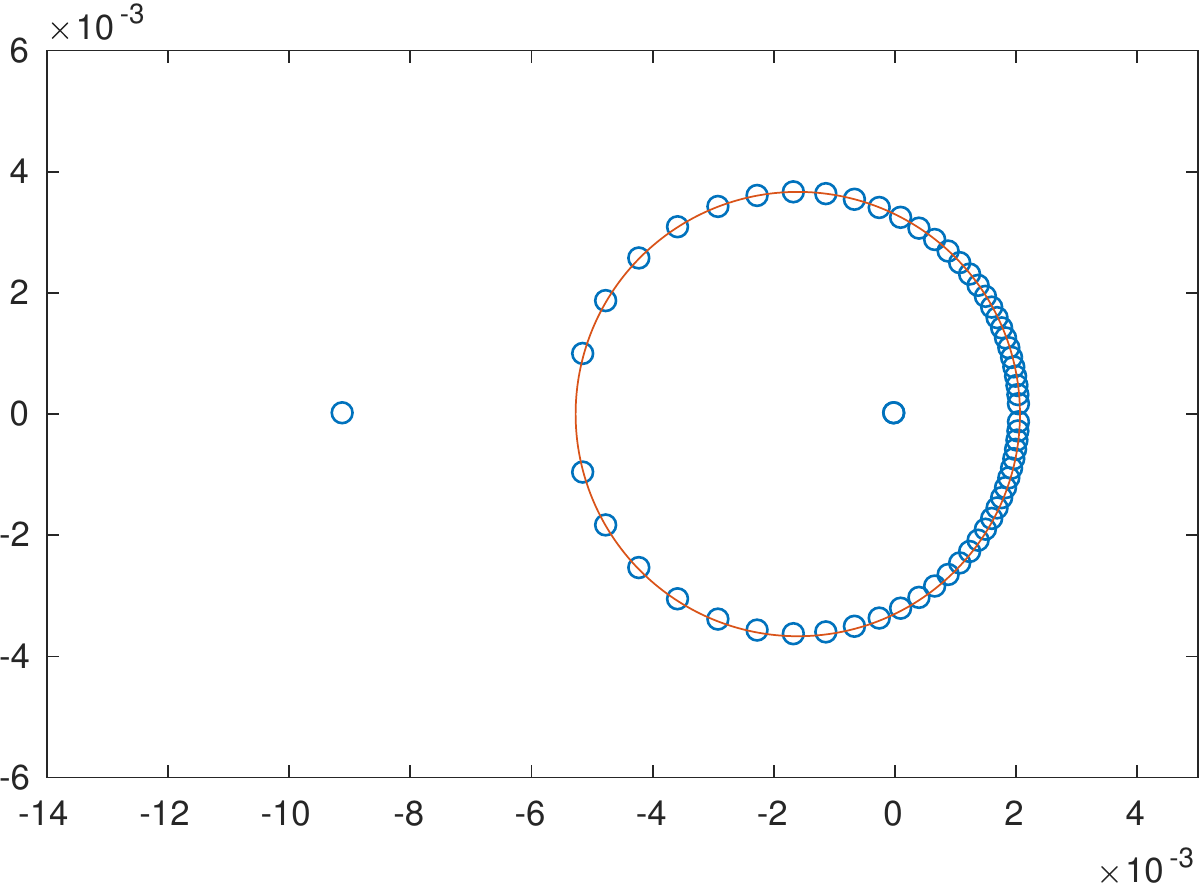}\\
\rotatebox{90}{\hskip 55pt $\mathrm{Im}(\mu)$}
\includegraphics[width=0.47\textwidth]{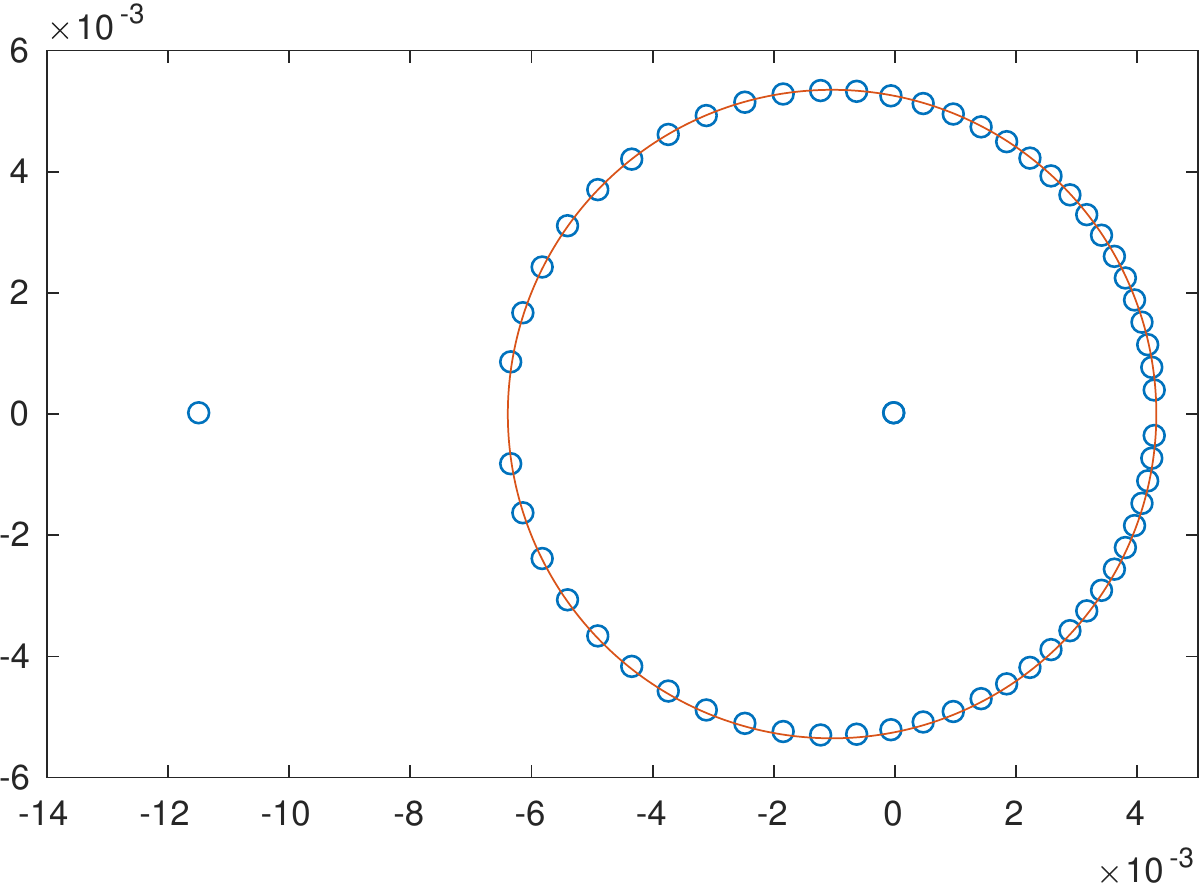}
\includegraphics[width=0.47\textwidth]{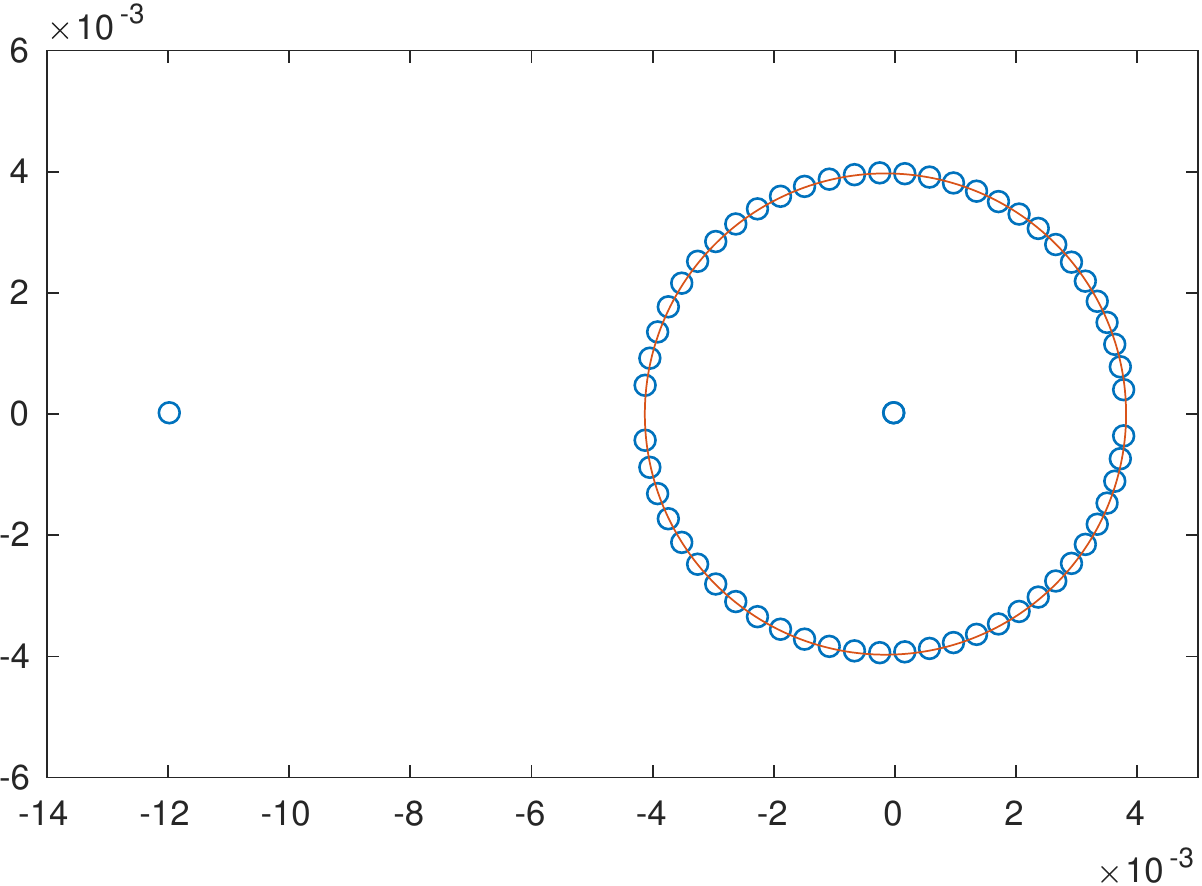}\\
\rotatebox{90}{\hskip 55pt $\mathrm{Im}(\mu)$}
\includegraphics[width=0.47\textwidth]{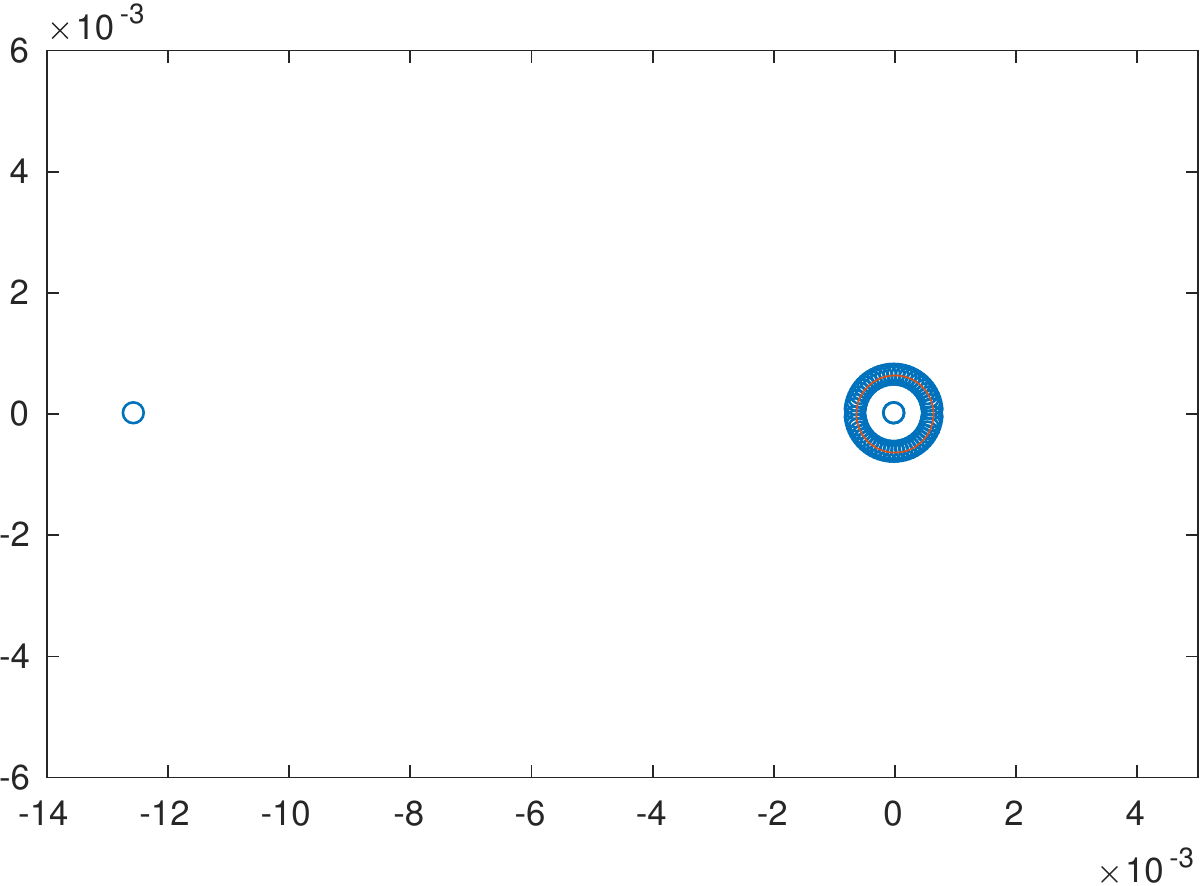}
\includegraphics[width=0.47\textwidth]{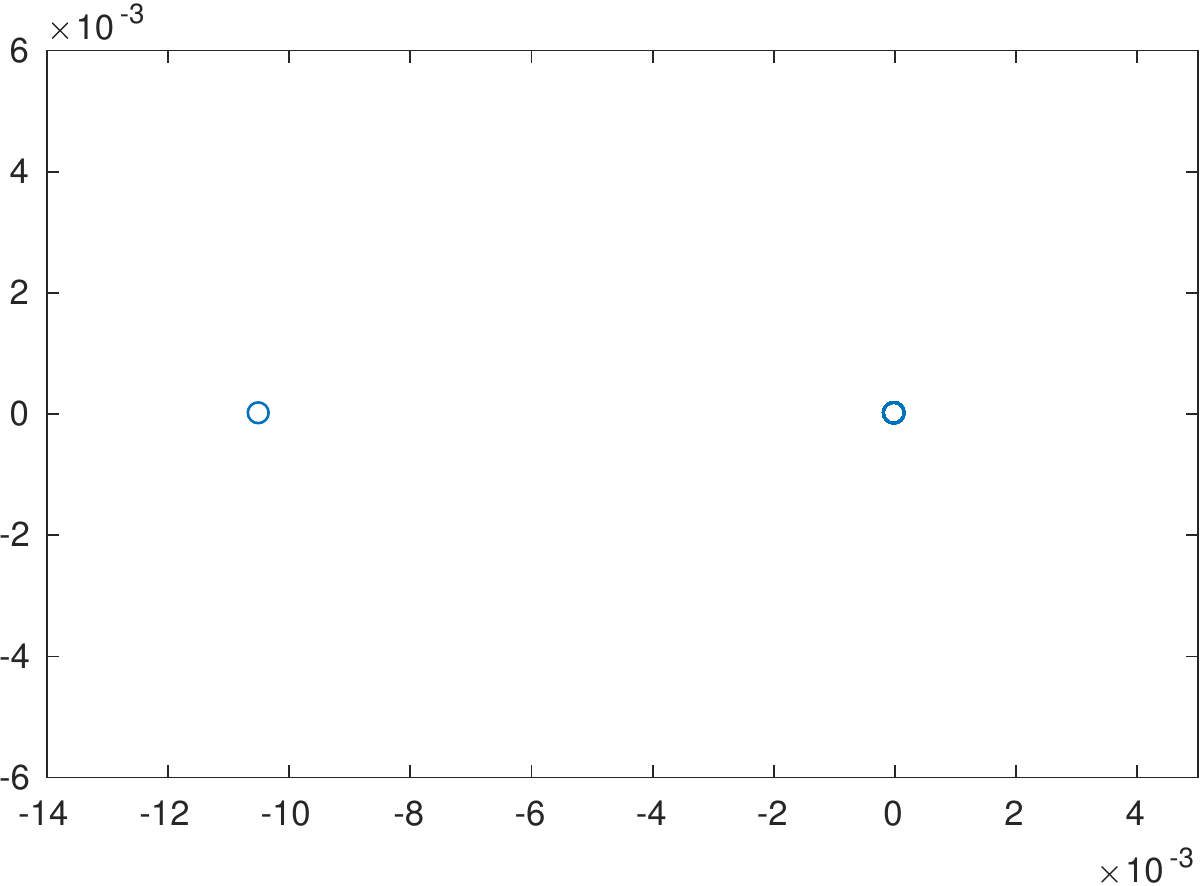}\\
\hfill $\hskip 25pt\mathrm{Re}(\mu)$ \hfill\hfill \hskip 15pt $\mathrm{Re}(\mu)$ \hfill\rule{0pt}{0pt}\\
\caption{Spectrum of \MG{the} \FK{iteration} matrix for $T = 100$, $L = 30$, $\Delta T/\Delta t = 50$,
 $\Delta t/\delta t=100$, $\alpha = 1$, and for $\sigma = -1/8, -1/4, -1/2, -1, -2, -16$, \FK{from top
 left to bottom right}.  \label{fig3.3}}
\end{figure}

\begin{figure}
\centering
\rotatebox{90}{\hskip 55pt $\mathrm{Im}(\mu)$}
\includegraphics[width=0.47\textwidth]{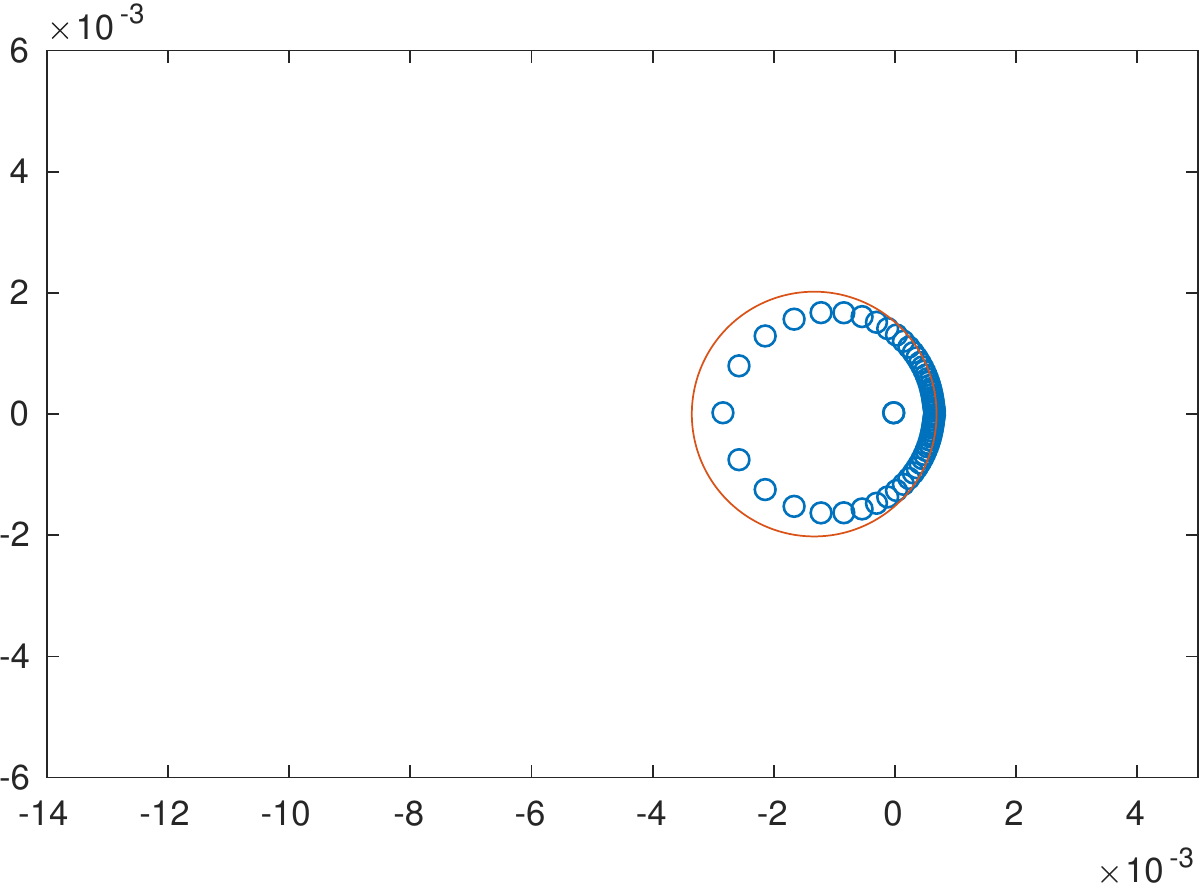}
\includegraphics[width=0.47\textwidth]{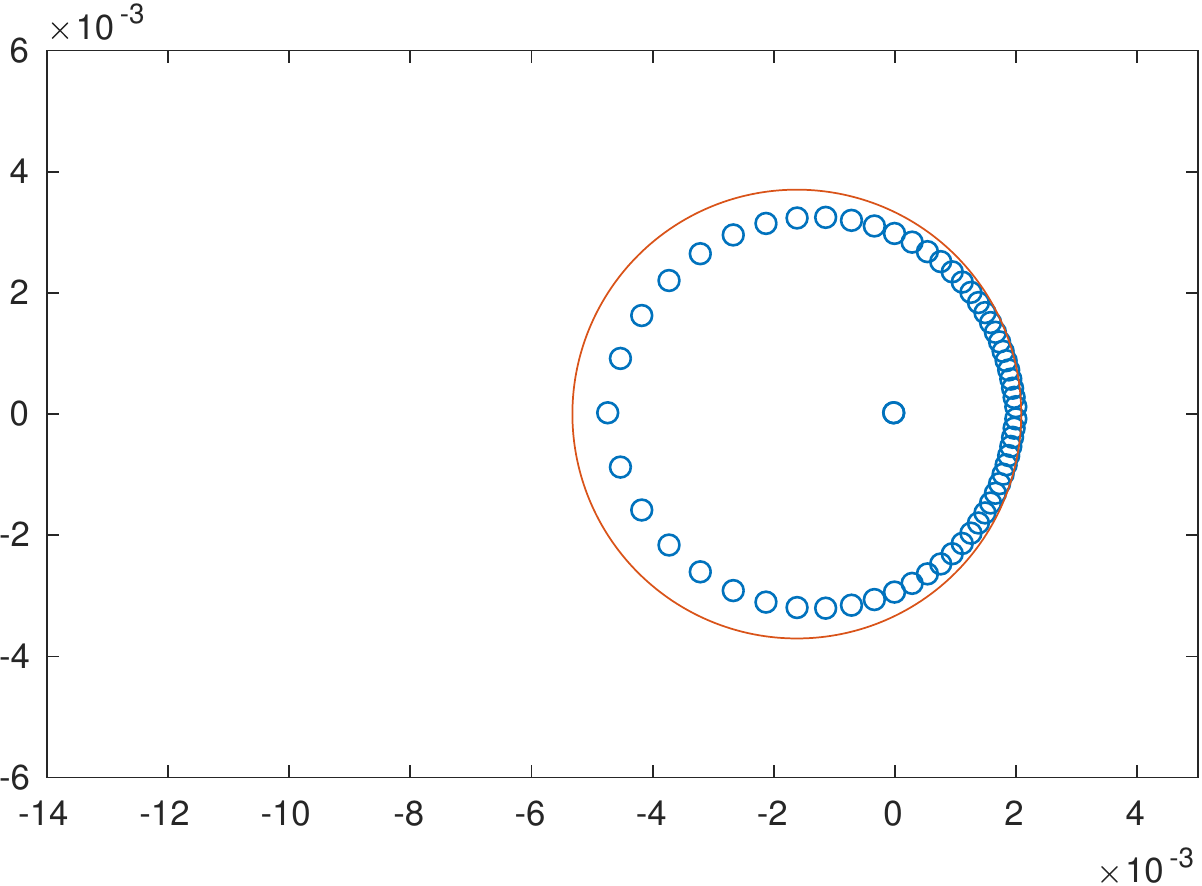}\\
\rotatebox{90}{\hskip 55pt $\mathrm{Im}(\mu)$}
\includegraphics[width=0.47\textwidth]{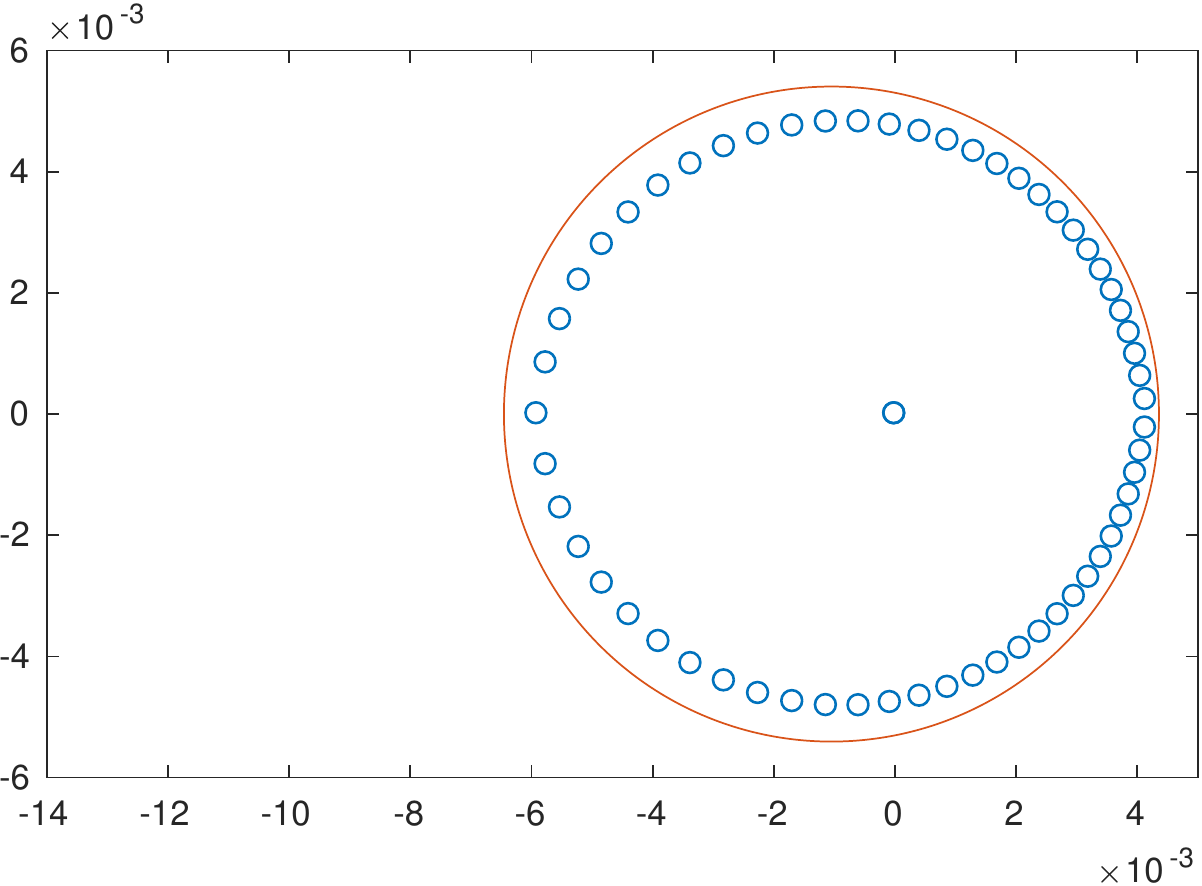}
\includegraphics[width=0.47\textwidth]{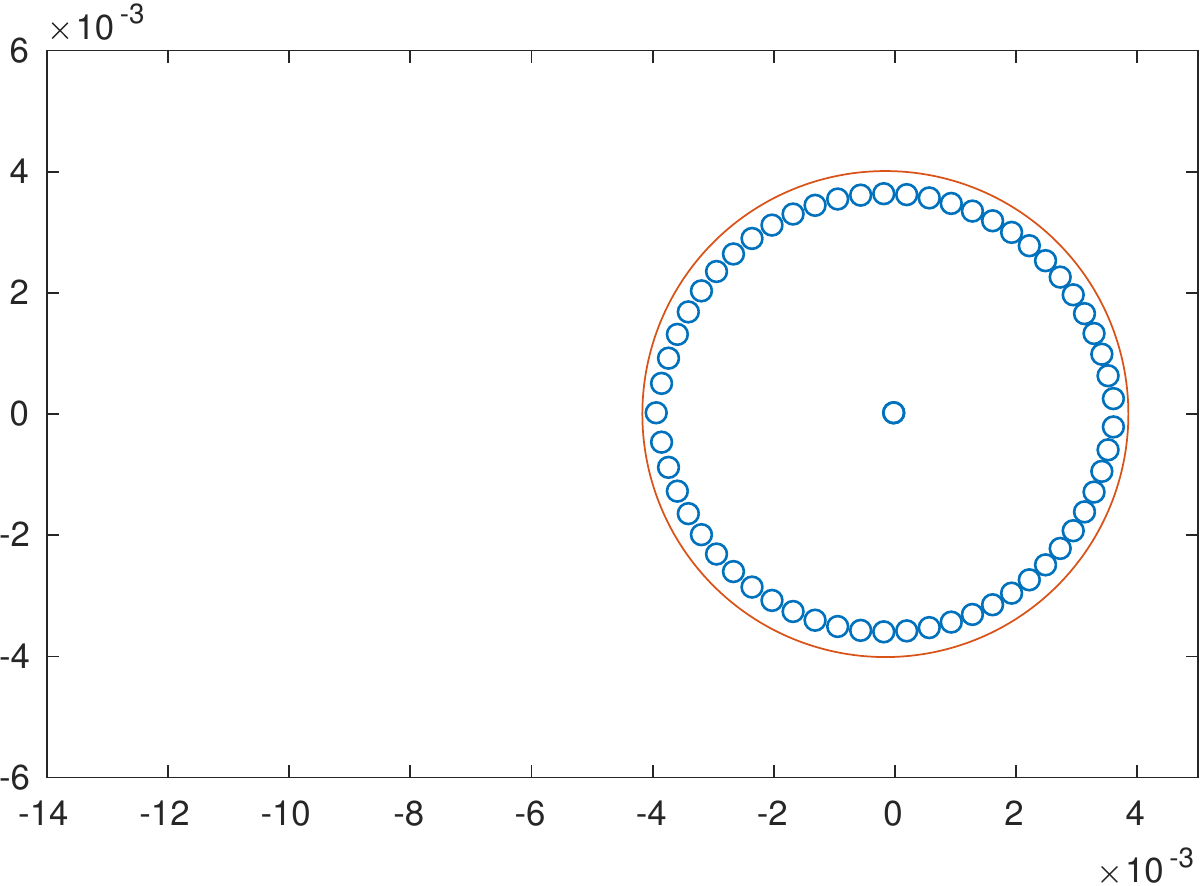}\\
\rotatebox{90}{\hskip 55pt $\mathrm{Im}(\mu)$}
\includegraphics[width=0.48\textwidth]{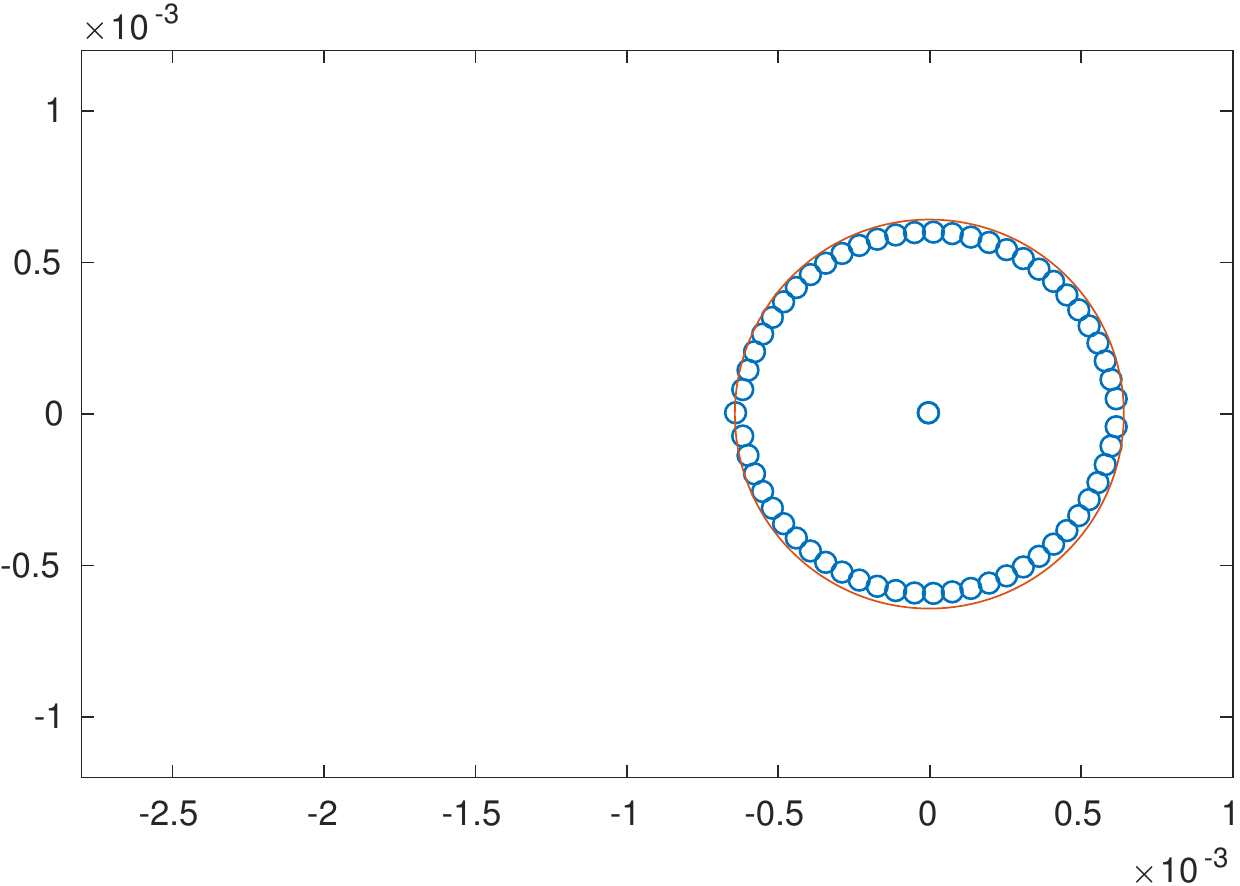}
\includegraphics[width=0.47\textwidth]{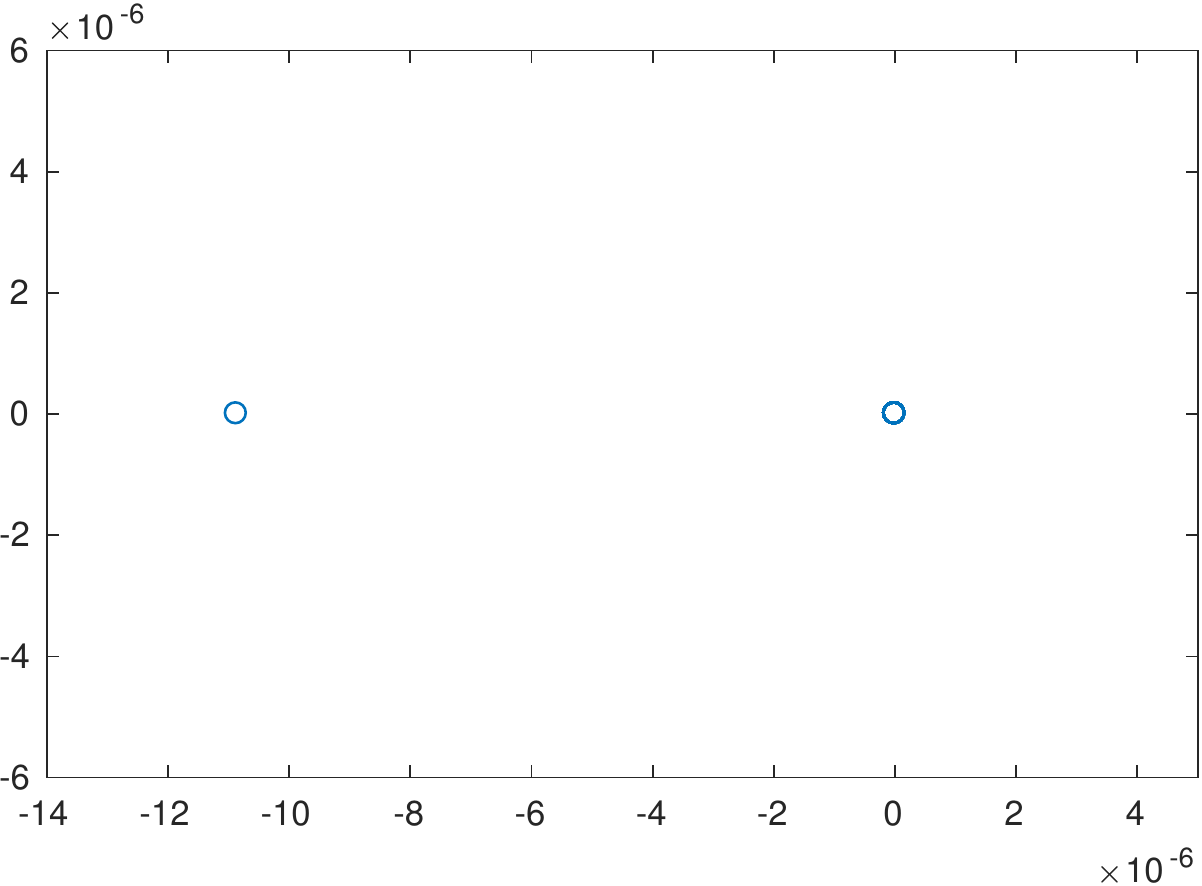}\\
\hfill $\hskip 25pt\mathrm{Re}(\mu)$ \hfill\hfill \hskip 15pt $\mathrm{Re}(\mu)$ \hfill\rule{0pt}{0pt}\\
\caption{Spectrum of \MG{the} \FK{iteration} matrix for $T = 100$, $L = 30$, $\Delta T/\Delta t = 50$,
 $\Delta t/\delta t=100$, $\alpha = 1000$, and for $\sigma = -1/8, -1/4, -1/2, -1, -2, -16$,
 \FK{from top left to bottom right}.
\label{fig3.4}}
\end{figure}

\begin{corollary}
Let $T$, $\Delta T$, $\Delta t$, $\delta t$, $\alpha$ and $\sigma$ be fixed. Then the spectral radius $\rho$ of the matrix $I-A_{\Delta t}^{-1}A_{\delta t}$ satisfies
\begin{equation}\label{mu_bound}
\rho \leq
\frac{|\delta\gamma|+\alpha\delta\beta(1+\beta)}{\gamma+\alpha(1-\beta^2)}.
\end{equation}
\end{corollary}
\FK{Note that the inequality \eqref{mu_bound} is valid for all $L>0$, i.e., regardless of whether the isolated eigenvalue $\mu^*$ exists.}
\begin{proof}
When the number of sub-intervals $L$ satisfies $L > \alpha L_0$, the spectral radius is determined by the isolated eigenvalue, which according to Theorem \ref{thm:isolated} is estimated by
$$ |\mu^*| < \frac{|\delta\gamma|+\alpha\delta\beta(1+\beta)}{\gamma+\alpha(1-\beta^2)}. $$
Otherwise, when $L \leq \alpha L_0$, all the eigenvalues lie within the bounding disc $D_\sigma$,
so no eigenvalue can be farther away from the origin than
$$ \mbox{Radius}(D_\sigma) + |\mbox{Center}(D_\sigma)| =
\frac{\delta\beta}{1-\beta^2} + \frac{\beta\delta\beta}{1-\beta^2} = \frac{\delta\beta}{1-\beta}. $$
A straightforward calculation shows that
$$ \frac{|\delta\gamma|+\alpha\delta\beta(1+\beta)}{\gamma+\alpha(1-\beta^2)}
> \frac{\delta\beta}{1-\beta} \quad \text{if and only if} \quad \beta + \frac{\gamma\delta\beta}{|\delta\gamma|} < 1,$$
which is true by Lemma \ref{constant_C}. \FK{Thus, the inequality \eqref{mu_bound} holds in both cases.}
%
%
\end{proof}


\begin{figure}
\centering
\includegraphics[width=0.45\textwidth]{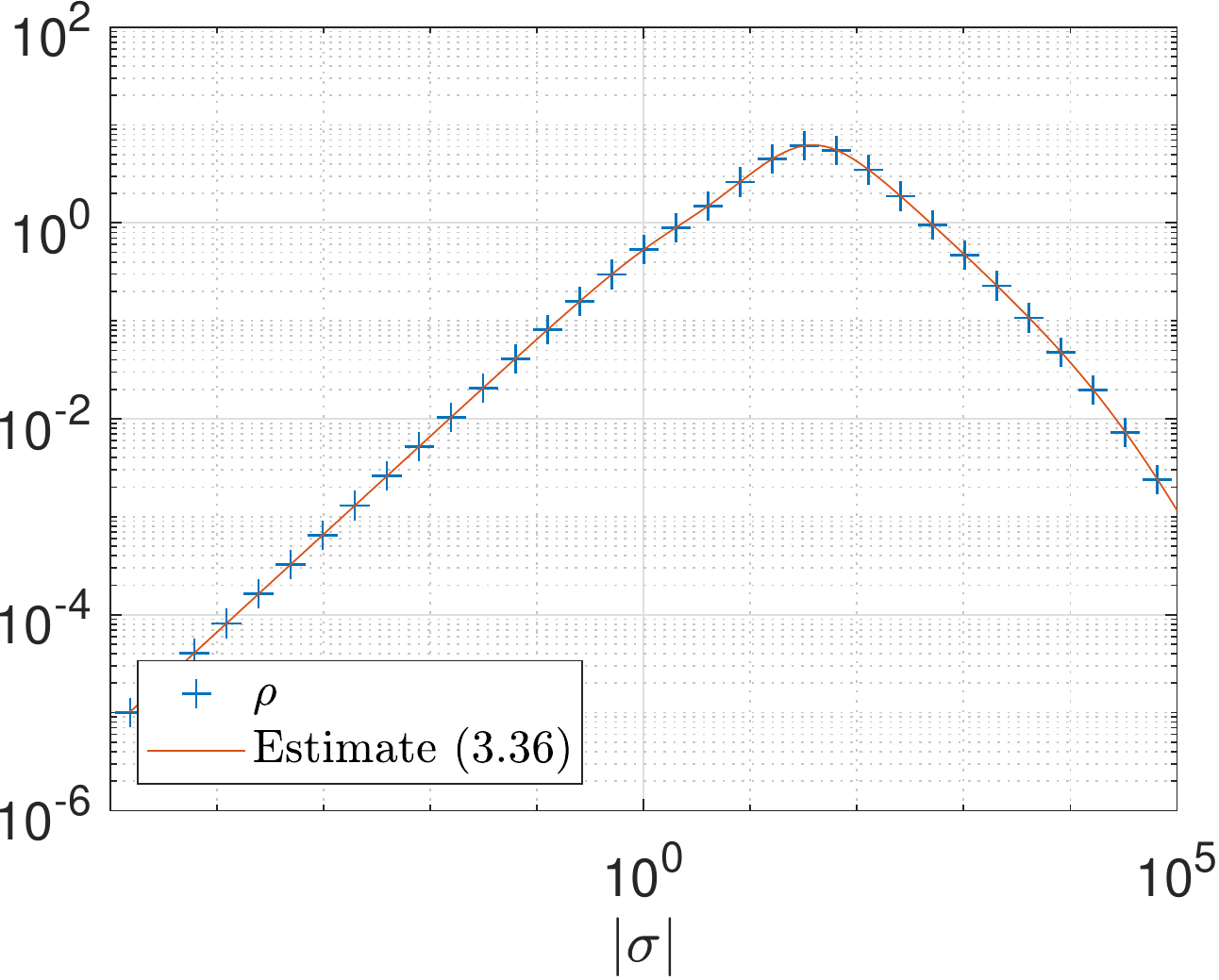}
\includegraphics[width=0.45\textwidth]{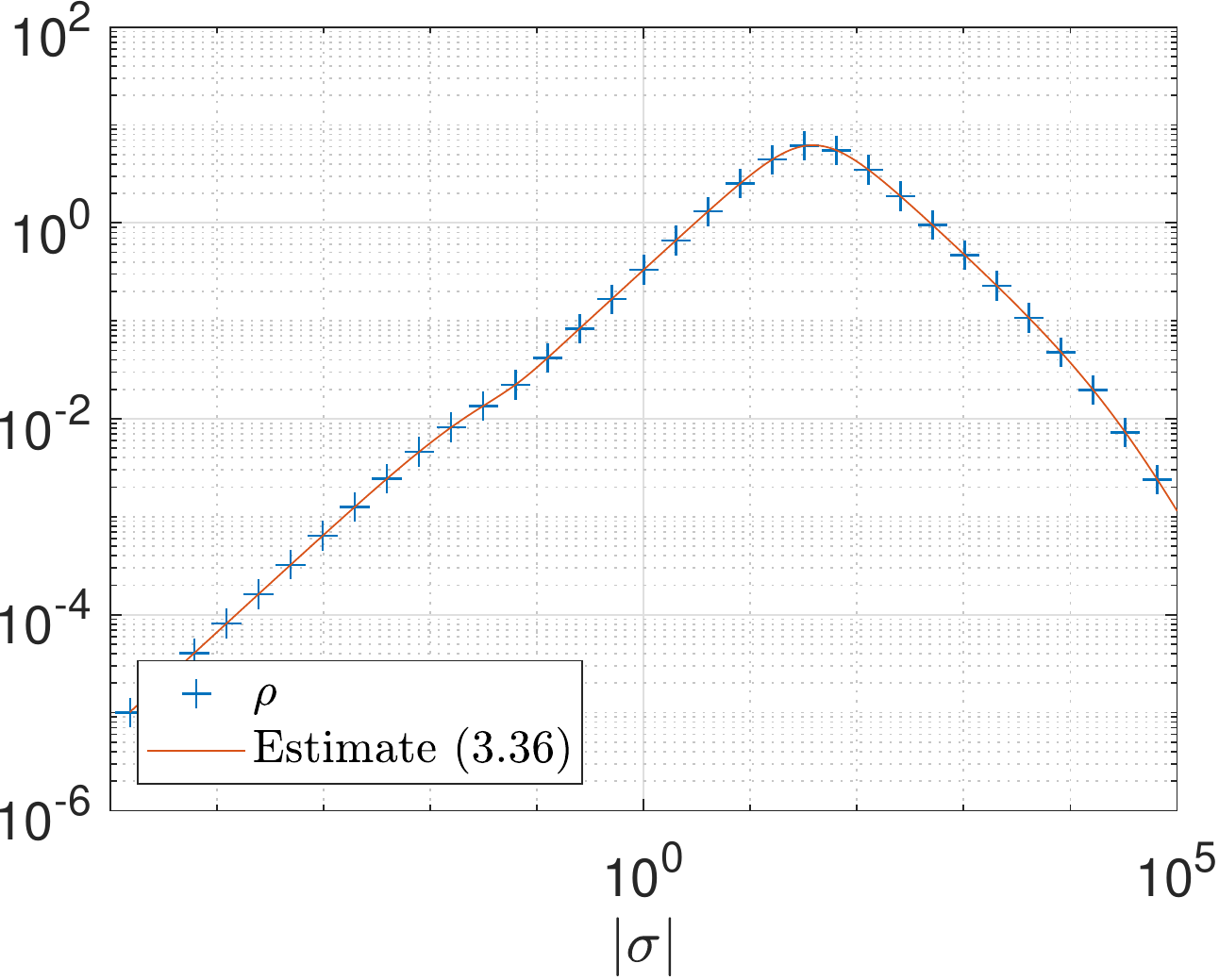}\\
\includegraphics[width=0.45\textwidth]{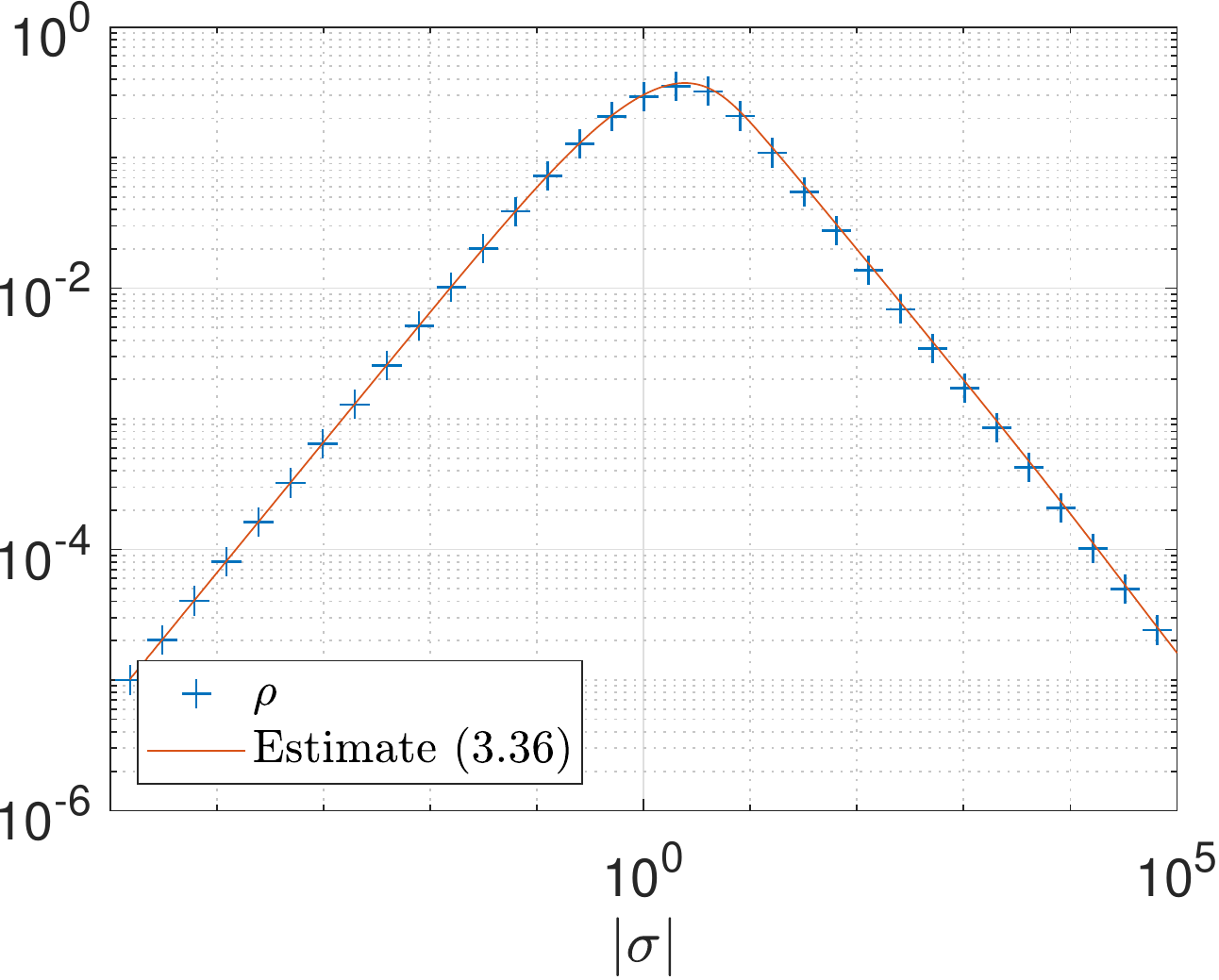}
\includegraphics[width=0.45\textwidth]{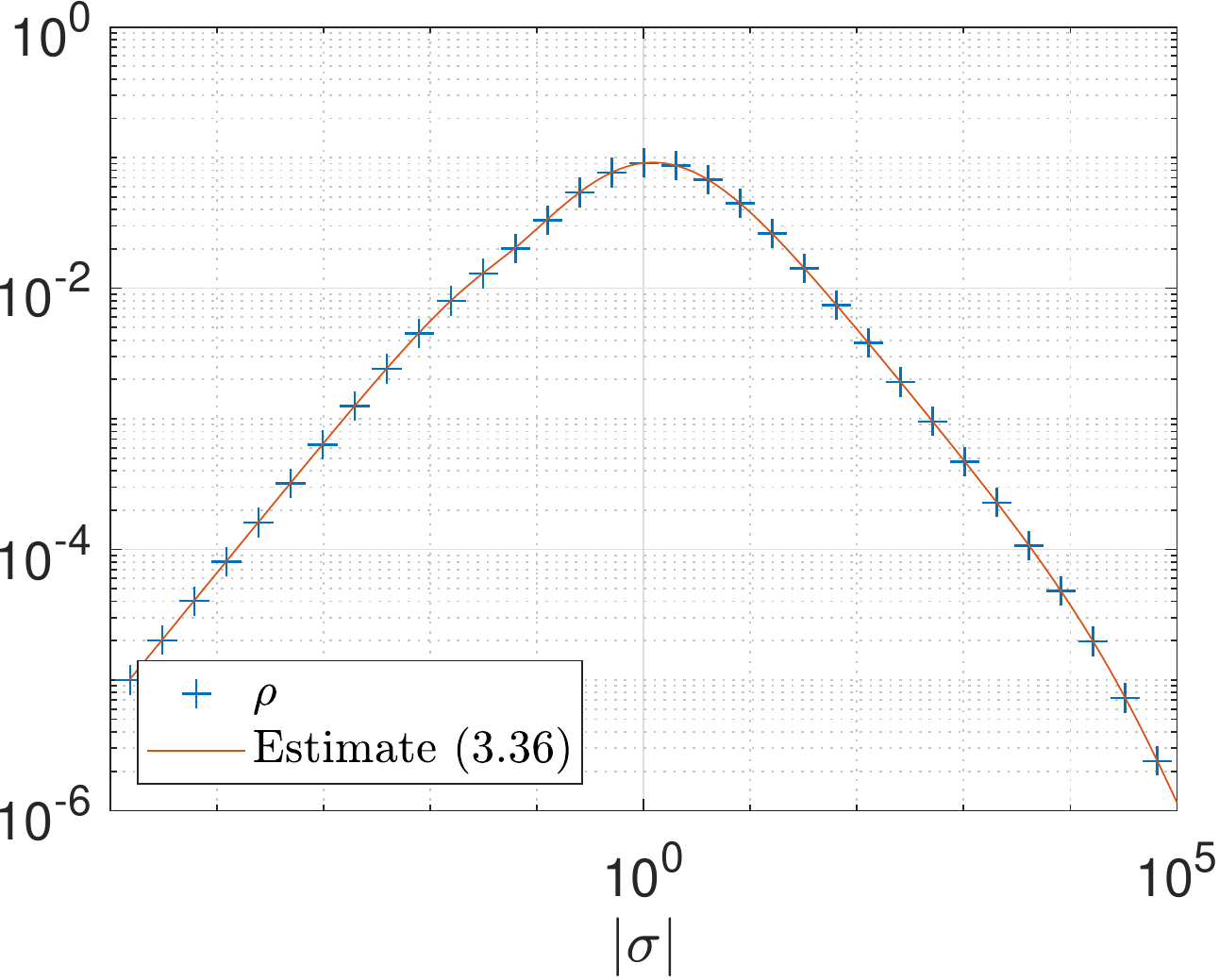}\\
\includegraphics[width=0.45\textwidth]{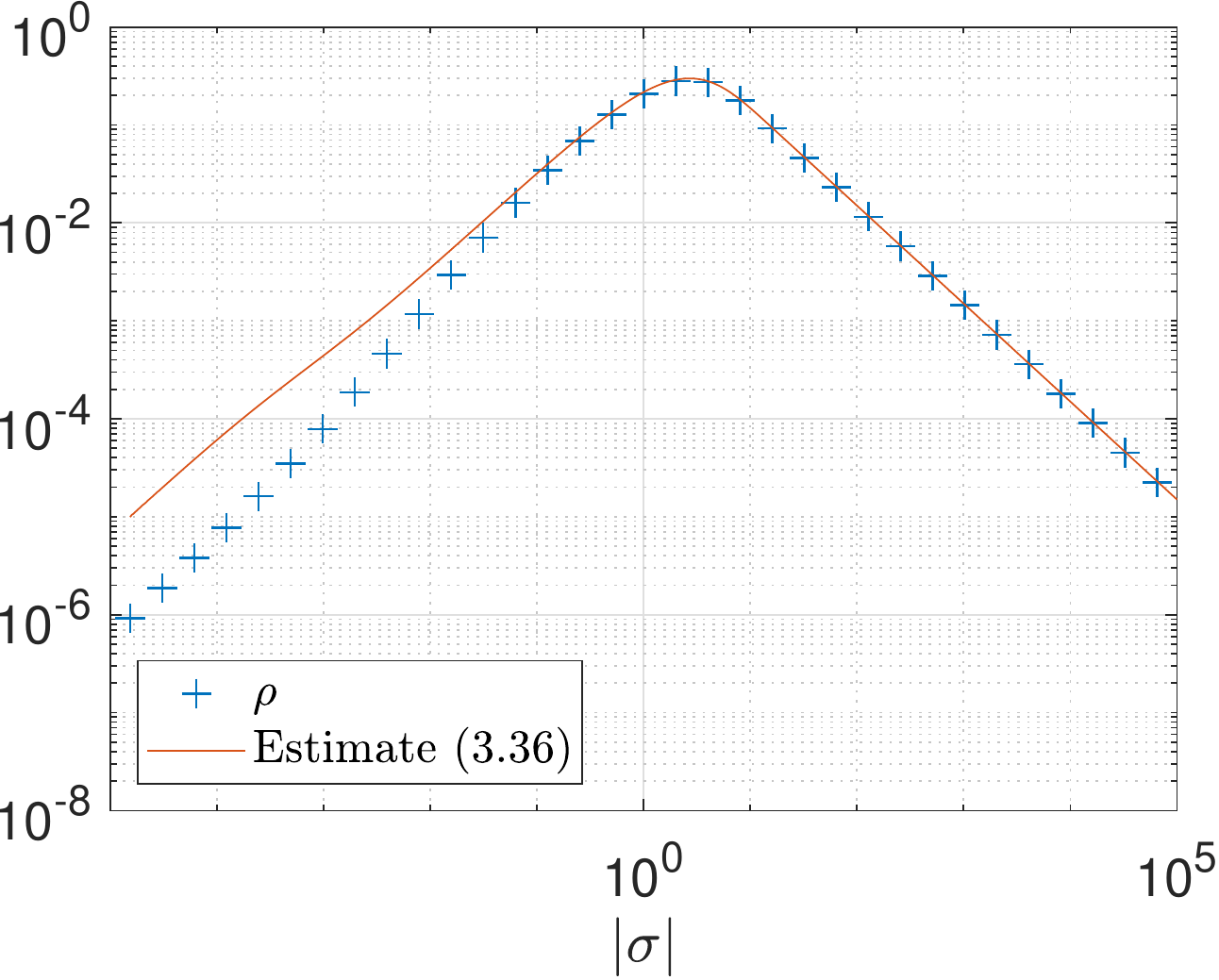}
\includegraphics[width=0.45\textwidth]{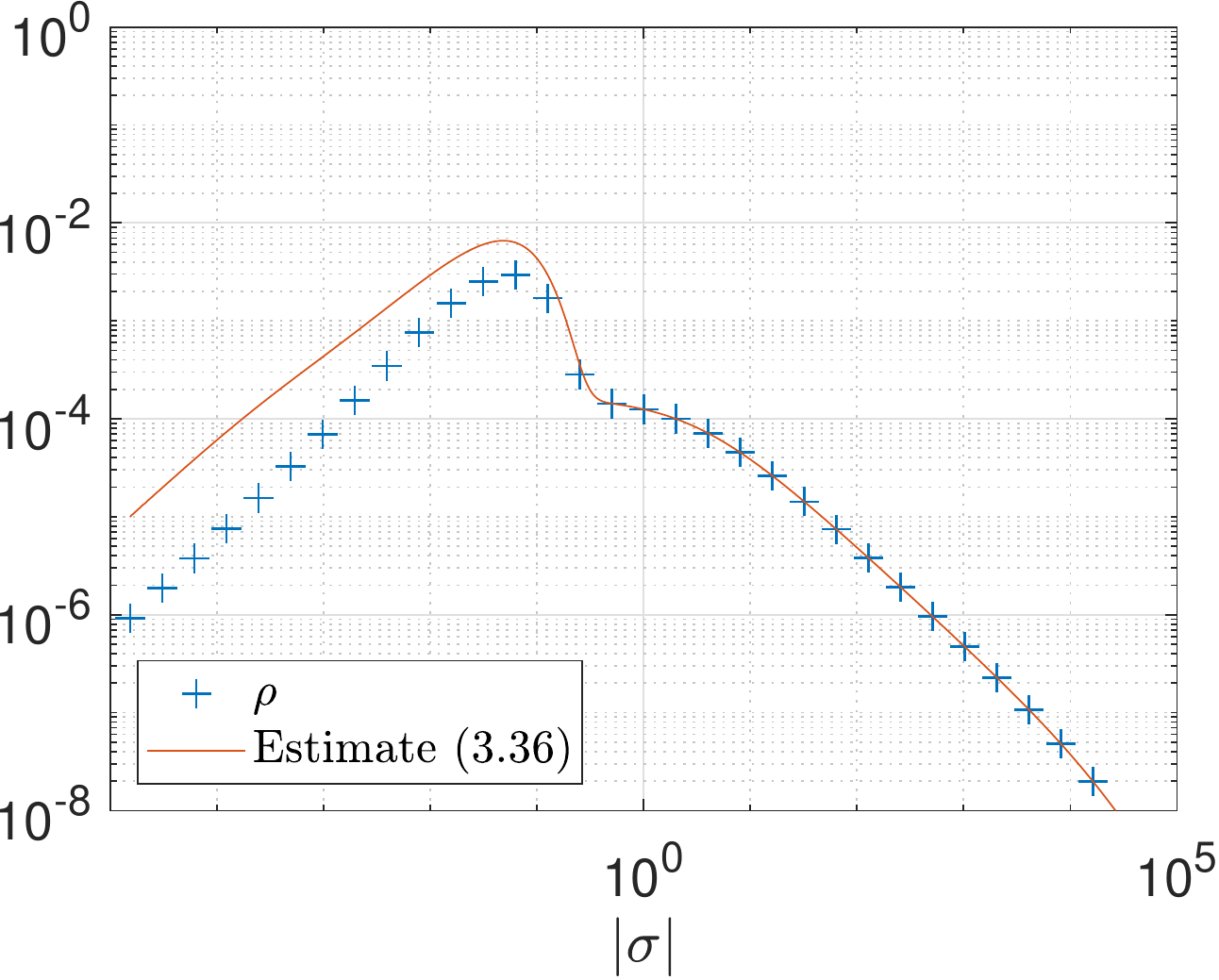}
\caption{Behaviour of $\mu_{\max}$ as a function of $\sigma$ for
$\alpha = 0.001, 1, 1000$ (top to bottom). Left: 150 subintervals, 1 coarse step per subinterval. Right: 3 subintervals, 50 coarse steps per subinterval. All examples use $T = 100$, $\Delta t = 2/3$ and $\Delta t/\delta t=10^{4}$. \label{fig:muhsal}}

\end{figure}

\FK{
The above corollary is of interest when we apply our ParaOpt method to a large system of ODEs (arising from the spatial discretization of a PDE, for example),
where the eigenvalues lie in the range $\sigma \in [-\sigma_{\max}, -\sigma_{\min}]$, with $\sigma_{\max} \to \infty$ when the
spatial grid is refined. As we can see from Figure \ref{fig:muhsal}, the upper bound follows the actual spectral radius rather
closely for most values of $\sigma$, and its maximum occurs roughly at the same value of $\sigma$ as
the one that maximizes the spectral radius. In the next two results,
we will use \FFK{the estimate \eqref{mu_bound} of}
\MG{the} spectral radius of $I-A_{\Delta t}^{-1}A_{\delta t}$ to derive a criterion for the convergence of the method.
}

\begin{lemma}\label{lem:dg_over_g}
Let $T$, $\Delta T$, $\Delta t$, $\delta t$ be fixed. Then for all $\sigma < 0$, we have
\begin{equation}\label{eq:dg_over_g}
\frac{|\delta\gamma|}{\gamma} \leq
1.58|\sigma|(\Delta t - \delta t), \qquad \frac{\delta\beta}{1-\beta} \leq 0.3.
\end{equation}
\end{lemma}
\begin{proof}
To bound $|\delta\gamma|/\gamma$, \FFK{we start by bounding a scaled version of the quantity.} We first use the definition of $\gamma$  \FFK{and $\gamma_{\delta t}$} (cf.~\eqref{eq:gamma_def}) to obtain
\begin{align*}
\frac{|\delta \gamma|}{\gamma}\cdot \frac{1}{|\sigma|(\Delta t - \delta t)} &= 
\frac{\gamma_{\delta t} - \gamma}{\gamma|\sigma|(\Delta t-\delta t)} \\
&=\frac{2+|\sigma|\Delta t}{(1-\beta^2)|\sigma|(\Delta t - \delta t)}
\left(\frac{1-\beta_{\delta t}^2}{2+|\sigma|\delta t} - \frac{1-\beta^2}{2+|\sigma|\Delta t}\right)\\
&= \frac{1-\beta_{\delta t}^2}{(2+\sigma\delta t)(1-\beta^2)} +
\frac{\beta^2-\beta_{\delta t}^2}{|\sigma|(\Delta t - \delta t)(1-\beta^2)} \FFK{=: A + B}.
\end{align*}
To estimate the terms \FFK{$A$ and $B$} above, we define the mapping
$$ h_{\Delta T}(\tau) := (1+|\sigma|\tau)^{-\Delta T/\tau}, $$
so that $\beta = h_{\Delta T}(\Delta t)$, $\beta_{\delta t} = h_{\Delta T}(\delta t)$.
Using the fact that $\ln(1+x) > \frac{x}{1+x}$ for $x > 0$ (see Lemma \ref{log} in Appendix \ref{appendix}),
we see that
$$h'_{\Delta T}(\tau) = h_{\Delta T}(\tau)\left[\frac{\Delta T}{\tau^2}\ln(1+|\sigma|\tau) - \frac{|\sigma|\Delta T}{\tau(1+|\sigma|\tau)}\right] > 0, $$
so $h_{\Delta T}$ is increasing. Therefore, we have
\begin{equation}\label{beta_ineq}
\lim_{\tau\to 0} h_{\Delta T}(\tau) = e^{-|\sigma|\Delta T} \leq \beta_{\delta t} \leq \beta \leq \frac{1}{1+|\sigma|\Delta T} = h_{\Delta T}(\Delta T).
\end{equation}
It then follows that
$$ \FFK{A :=} \frac{1-\beta_{\delta t}^2}{(2+\sigma\delta t)(1-\beta^2)} \leq
\frac{1-e^{-2|\sigma|\Delta T}}{(2+\sigma\delta t)(1-(1+|\sigma|\Delta T)^{-2})}
\leq \frac{(1-e^{-2|\sigma|\Delta T})(1+|\sigma|\Delta T)^2}{2|\sigma|\Delta T(2+|\sigma|\Delta T)}.$$
The last quotient is a function in $|\sigma|\Delta T$ only, whose maximum over all $|\sigma|\Delta T > 0$
is approximately $0.5773 < 0.58$; \FFK{therefore, we have
$$A \leq 0.58. $$
}
For the second term, we use the mean value theorem and the fact that $\beta^2 =
h_{2\Delta T}(\Delta t)$, $\beta_{\delta t}^2 = h_{2\Delta T}(\delta t)$ to obtain
$$ \beta^2 - \beta_{\delta t}^2 = (\Delta t - \delta t)h'_{2\Delta T}(\tau^*) $$
for some $\delta t < \tau^* < \Delta t$, with
$$h'_{2\Delta T}(\tau) = h_{2\Delta T}(\tau)\left[\frac{2\Delta T}{\tau^2}\ln(1+|\sigma|\tau) - \frac{2|\sigma|\Delta T}{\tau(1+|\sigma|\tau)}\right]. $$
Using the fact that $\ln(1+x) \leq x$ for all $x \geq 0$, we deduce that
$$ h'_{2\Delta T}(\tau^*) \leq h_{2\Delta T}(\tau^*) \frac{2|\sigma|^2\Delta T}{1+|\sigma|\tau^*}
\leq \frac{2\beta^2|\sigma|^2\Delta T}{1+|\sigma|\delta t}, $$
so that
$$ \FFK{B:=}\frac{\beta^2-\beta_{\delta t}^2}{|\sigma|(\Delta t - \delta t)(1-\beta^2)}
\leq \frac{2|\sigma|\Delta T}{(1+|\sigma|\Delta T)^2}\cdot\frac{(1+|\sigma|\Delta T)^2}{|\sigma|\Delta T(2+|\sigma|\Delta T)}\leq 1.$$
\FFK{Combining the estimates for $A$ and $B$ and multiplying by $|\sigma|(\Delta t - \delta t)$} gives the first inequality in \eqref{eq:dg_over_g}. For the second
inequality, we use \eqref{beta_ineq} to obtain
$$ \frac{\beta - \beta_{\delta t}}{1-\beta} \leq \frac{(1+|\sigma|\Delta T)^{-1} - e^{-|\sigma|\Delta T}}{1-(1+|\sigma|\Delta T)^{-1}} = \frac{1 - (1+|\sigma|\Delta T)e^{-|\sigma|\Delta T}}{|\sigma|\Delta T}. $$
This is again a function in a single variable $|\sigma|\Delta T$, whose maximum over all $|\sigma|\Delta T > 0$ is approximately $0.2984 < 0.3$.
\end{proof}

\begin{theorem}\label{lin_conv_thm}
Let $\Delta T$, $\Delta t$, $\delta t$ and $\alpha$ be fixed. Then for all $\sigma < 0$, the spectral radius of $I-A_{\Delta t}^{-1}A_{\delta t}$ satisfies
\begin{equation}\label{eq:spec_radius}
\max_{\sigma < 0} \rho(\sigma) \leq \frac{0.79\Delta t}{\alpha+\sqrt{\alpha\Delta t}} + 0.3.
\end{equation}
Thus, if $\alpha > \FFK{0.4544}\Delta t$, then the linear \FK{ParaOpt} algorithm \eqref{eq:scheme} converges.
\end{theorem}

\begin{proof}
\FFK{Starting with the}
spectral radius estimate \eqref{mu_bound}, \FFK{we divide the numerator and denominator by $\gamma$, then substitute its definition in \eqref{eq:gamma_def} to obtain} 
\begin{align*}
\rho(\sigma) &< \frac{|\delta\gamma|+ \alpha\delta\beta(1+\beta)}{\gamma+\alpha(1-\beta^2)}
= \frac{\frac{|\delta\gamma|}{\gamma} + \frac{\delta\beta}{1-\beta}\alpha|\sigma|(2+|\sigma|\Delta t)}{1+\alpha|\sigma|(2+|\sigma|\Delta t)}\\
&\leq \frac{|\delta\gamma|}{\gamma(1+\alpha|\sigma|(2+|\sigma|\Delta t)\FK{)}}+ \frac{\delta\beta}{1-\beta}.
\end{align*}
Now, by Lemma \ref{lem:dg_over_g}, the first term is bounded above by
$$ f(\sigma) := \frac{1.58|\sigma|\Delta t}{1+\alpha|\sigma|(2+|\sigma|\Delta t)}, $$
whose maximum occurs at $\sigma^* = -1/\sqrt{\alpha\Delta t}$ with
$$ f(\sigma^*) = \frac{0.79\Delta t}{\sqrt{\alpha\Delta t} + \alpha}. $$
Together with the estimate on $\delta\beta/(1-\beta)$ in Lemma \ref{lem:dg_over_g}, this proves \eqref{eq:spec_radius}. Thus, a sufficient condition for the method \eqref{eq:scheme} to converge
can be obtained by solving the inequality
$$ \frac{0.79\Delta t}{\alpha+\sqrt{\alpha\Delta t}} + 0.3 < 1. $$
This is a quadratic equation in $\sqrt{\alpha}$; solving it leads to $\alpha > \FFK{0.4544}\Delta t$, as required.
\end{proof}
$\quad$\\
\par
In Figure~\ref{AlphaAsymptotics}, we show the maximum spectral radius of $I-A_{\Delta t}^{-1}A_{\delta t}$ over all negative $\sigma$ for different values of $\alpha$ for a model decomposition with
$T = 100$, 30 subintervals, one coarse time step per sub-interval, and a refinement ratio of $10^4$
between the coarse and fine grid. We see in this case that the estimate \eqref{eq:spec_radius} is indeed
quite accurate.\\

\begin{figure}
\centering
\includegraphics[width=0.7\textwidth]{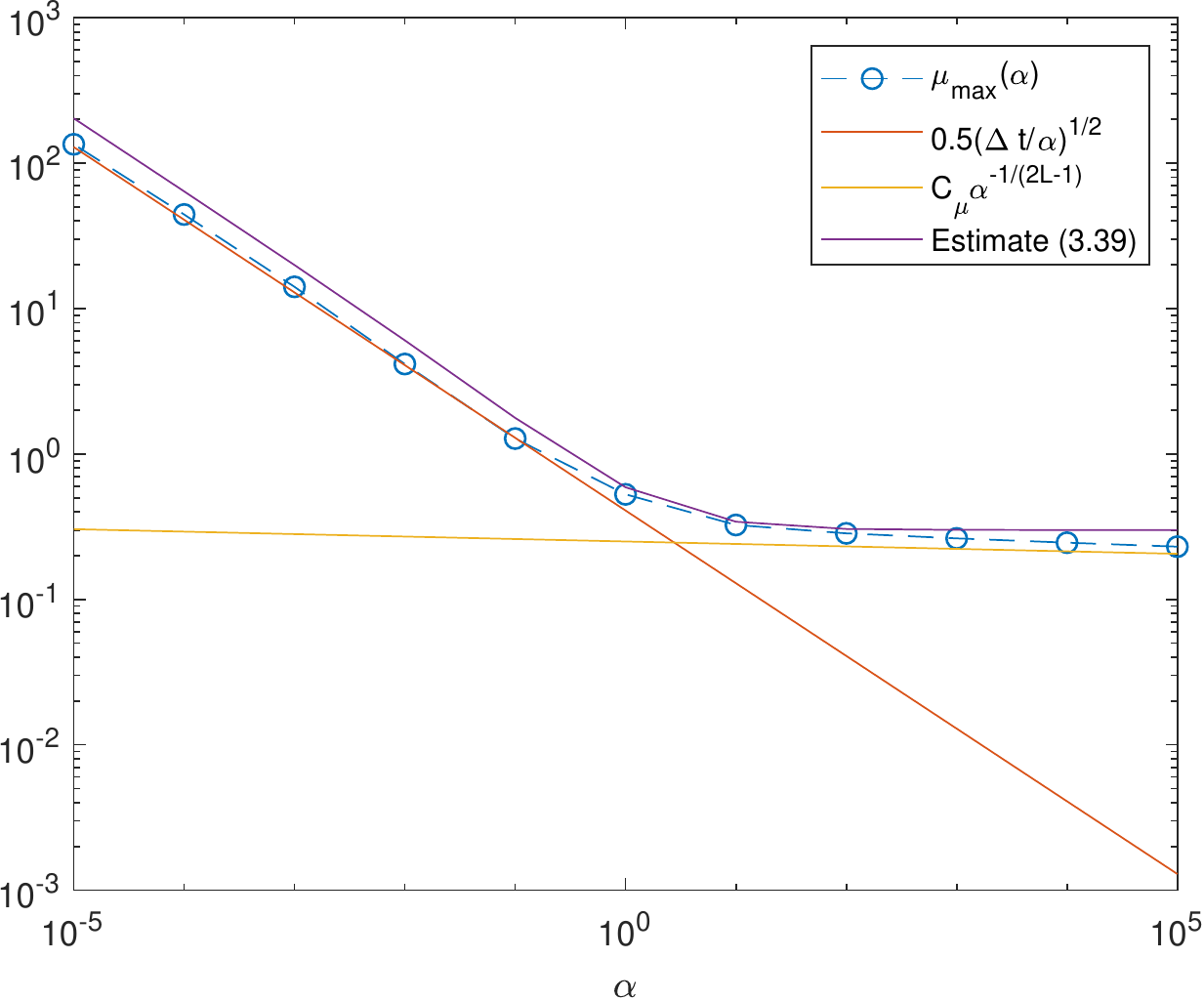}
\caption{Behaviour of $\max_{\sigma < 0} \rho(\sigma)$  as a function of $\alpha$, $T = 100$, $L = 30$, $\Delta T = \Delta t$, $\Delta t/\delta t =10^{-4}$. \FFK{The data for $\mu_{\max}(\alpha)$ has been generated by solving the generalized eigenvalue problem \eqref{eq:eig_prob} using {\tt eig} in {\sc Matlab}}. \label{AlphaAsymptotics}}

\end{figure}

\par \noindent \emph{Remarks.}
\begin{enumerate}[leftmargin=*]
\item (Dependence on $\alpha$) Theorem \ref{lin_conv_thm} states that in order to guarantee convergence, one should make sure that the coarse time step $\Delta t$ is sufficiently small relative to $\alpha$. In that case, the method converges.
\item (Weak scalability)
 Note that the estimate \eqref{eq:spec_radius} depends on the coarse time step $\Delta t$, but
 not explicitly on the number of sub-intervals $L$.
One may then consider weak scalability, \JS{ i.e. cases where the problem size per processor is fixed\footnote{On the contrary, strong scalability deals with cases where the total problem size is fixed.}},
under two different regimes: (i) keeping the sub-interval length
$\Delta T$ and refinement ratios $\Delta T/\Delta t$, $\Delta t/\delta t$ fixed, such that adding
subintervals increases the overall time horizon $T = L\Delta T$; and
(ii) keeping the time horizon $T$ fixed and refinement ratios $\Delta T/\Delta t$, $\Delta t/\delta t$ fixed, such that adding sub-intervals decreases their length $\Delta T = T/L$. In the first case, $\Delta t$
remains fixed, so the bound \eqref{eq:spec_radius} remains bounded as $L\to\infty$. In the second case, $\Delta t \to 0$ as $L\to \infty$, so in fact \eqref{eq:spec_radius} decreases to $0.3$ as $L\to \infty$.
Therefore, the method is weakly scalable \emph{under both regimes}.
\item (Contraction rate for high and low frequencies)
Let $\alpha > 0$ be fixed, and let $\rho(\sigma)$ be the spectral radius of $I - A_{\Delta t}^{-1}A_{\delta t}$ as a function of $\sigma$ given by \eqref{mu_bound}. Then for $\Delta t/\delta t \geq 2$,
an asymptotic expansion shows that we have
$$ \rho(\sigma) = \begin{cases}
|\sigma|(\Delta t - \delta t) + O(|\sigma|^2) & \text{as $|\sigma| \to 0$},\\
\frac{1}{|\sigma|\Delta t}+O(|\sigma|^{-2}) & \text{as $|\sigma|\to \infty$ if $\Delta T = \Delta t$,}\\
\frac{1}{\alpha\delta t}|\sigma|^{-2} + O(|\sigma|^{-3}) & \text{as $|\sigma| \to \infty$ if $\Delta T/\Delta t \geq 2$.} \end{cases} $$
In other words, the method reduces high and low frequency error modes  very quickly, and the overall contraction
rate is dominated by mid frequencies (where ``mid'' depends on $\alpha$, $\Delta t$, etc). This is also visible in Figure \ref{fig:muhsal}, where $\rho$ attains its maximum at $|\sigma| = O(1/\sqrt{\alpha})$ and decays quickly for both large and small $|\sigma|$.\\
\end{enumerate}

Finally, we note that for the linear problem, it is possible to use Krylov acceleration to solve \FK{for the fixed point of}
\eqref{eq:scheme}, even when the spectral radius is greater than 1.
\FK{However, the goal of this linear analysis is to use it as a tool for studying the asymptotic behaviour of the \emph{nonlinear} method \eqref{OurMethod}; since a contractive fixed point map must have a
Jacobian with spectral radius less than 1 at the fixed point, Theorem \ref{lin_conv_thm} shows
which conditions are sufficient to ensure asymptotic convergence of the nonlinear ParaOpt method.}

\section{Numerical results}\label{Sec4}
\JS{In the previous section, we have presented numerical examples related to the
efficiency of our bounds with respect to $\sigma$ and $\alpha$. We now
study in more detail the quality of our bounds with respect to the
discretization parameters. We complete these experiments with a
nonlinear example and a PDE example.}

\subsection{Linear scalar \MG{ODE}: sensitivity with respect to the discretization parameters}
In this part, we consider the case where $\alpha=1$, $\sigma=-16$ and $T=1$ and investigate the dependence of the spectral radius of $I-A_{\Delta t}^{-1}A_{\delta t}$ when $L$, $\Delta t$, $\MG{\delta} t$ vary.

We start with variations in $\Delta t$ and $\delta t$, and a fix\MG{ed} number of sub-intervals $L=10$.
In this way, \MG{we} compute the spectral radius of $I-A_{\Delta t}^{-1}A_{\delta t}$ for three cases: first with a fixed $\Delta t=10^{-4}$ and $\delta t=\frac{\Delta t}{2^k}$, $k=1,\MG{\ldots},15$\MG{;} then  with a fixed $\delta t=10^{-2}\cdot 2^{-20}$ and $\Delta t=2^{-k}$, $k=0,\MG{\ldots},20$; and finally with a fixed ratio $\frac{\delta t}{\Delta t}=10^{-2}$ with $
\Delta t= 2^k$, $k=1,\MG{\ldots},15$. The results are shown in Fig\MG{ure}~\ref{figdtDtDt}.
\begin{figure}
\centering
\includegraphics[width=0.48\textwidth]{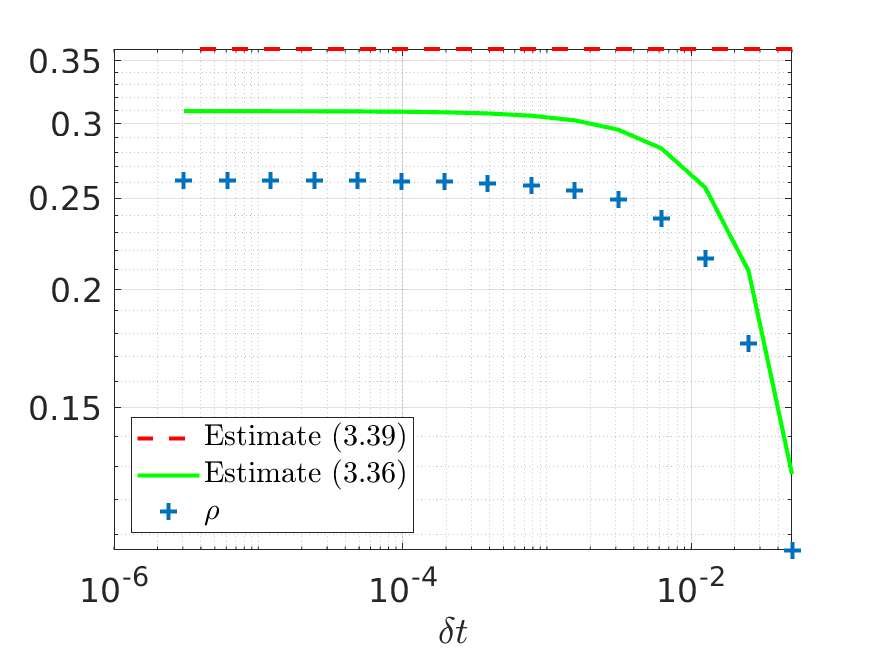}
\includegraphics[width=0.48\textwidth]{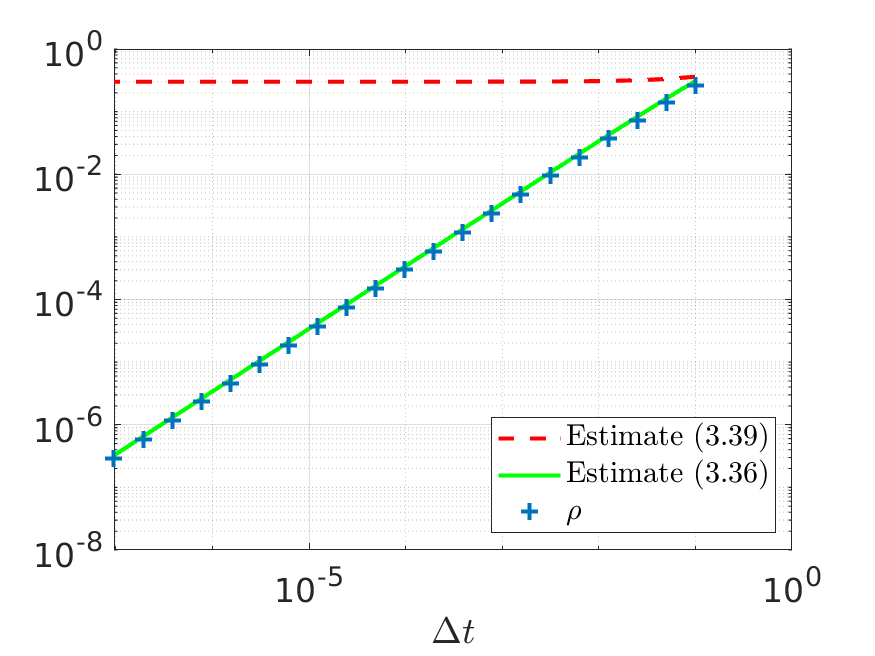}\\
\includegraphics[width=0.48\textwidth]{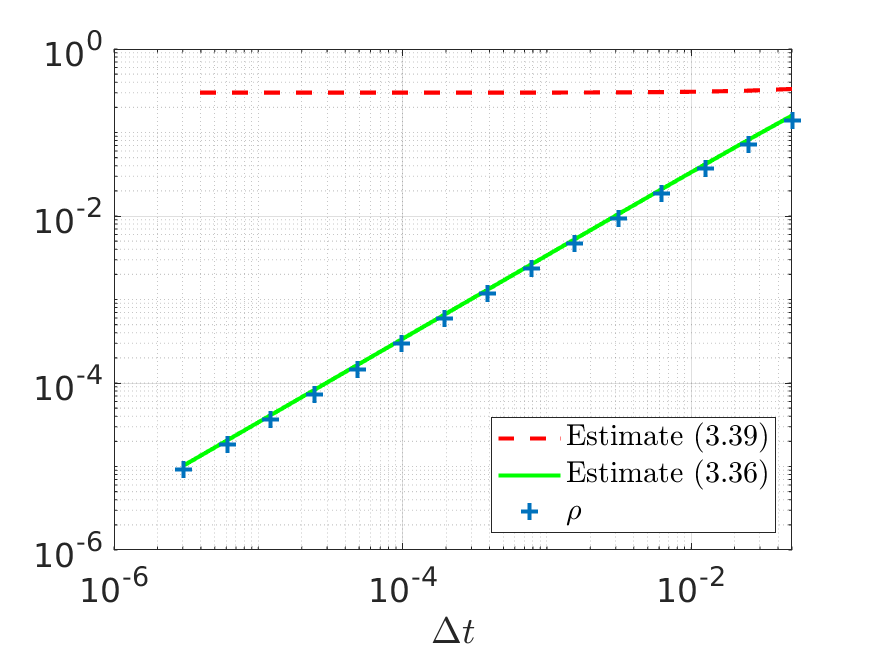}
\caption{Spectral radius of \MG{the} preconditioned matrix. 
 Top left:  varying $\delta t$ (with fixed $\Delta t$), top right: varying $\Delta t$  (with fixed $\delta t$), bottom: varying $\Delta t$ (with fixed $\frac{\delta t}{\Delta t}$).  \label{figdtDtDt}}
\end{figure}
In all cases, we observe a very good agreement between the estimate obtained in~\eqref{mu_bound} and the true spectral radius. \FFK{Note that the largest possible $\Delta t$ for this problem is when
$\Delta t$ equals the length of the sub-interval, i.e., when $\Delta t = \Delta T = 0.1$\JJJS{. For this $\Delta t$},  the estimates (3.36) and (3.39) are very close to each other, because (3.39) is obtained from (3.36) by 
making $\Delta t$ as large as possible, i.e., by letting $\Delta t = \Delta T$.}


We \MG{next} study the scalability properties of \MMG{ParaOpt}. More
precisely, we examine the behaviour of the spectral radius of \MG{the}
preconditioned matrix when the number of subintervals $L$ varies. In
order to fit with the paradigm of numerical efficiency, we set $\Delta
T=\Delta t$ which corresponds somehow to a coarsening limit. We
consider two cases: \FK{the first case uses} a fixed value of $T$, namely $T=1$, and
the second case uses $T=L\Delta T$ \FK{ for the fixed value of $\Delta T = 1$}. The results are shown in Fig\MG{ure}~\ref{scal}.
\begin{figure}
  \centering
\includegraphics[width=0.48\textwidth]{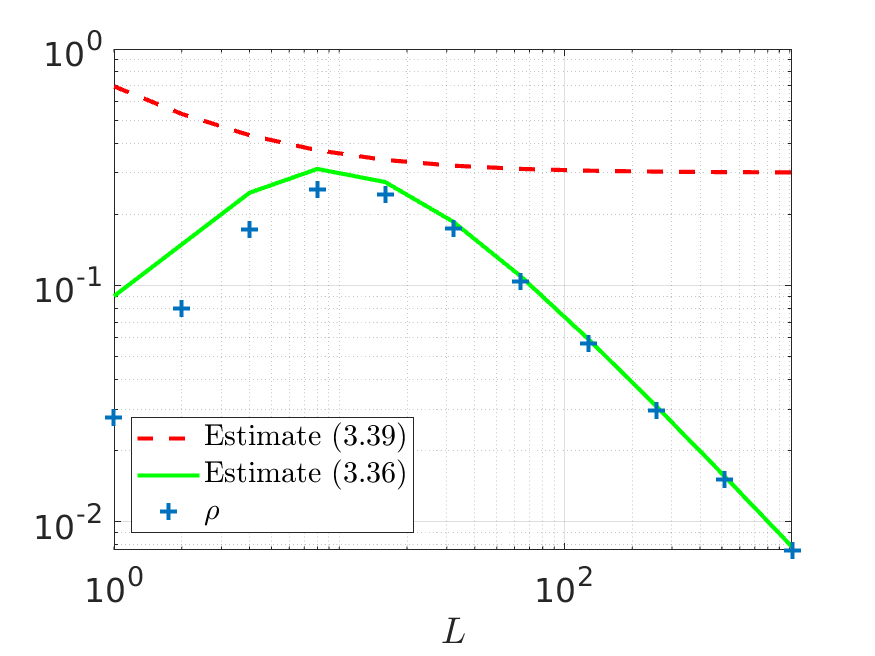}
\includegraphics[width=0.48\textwidth]{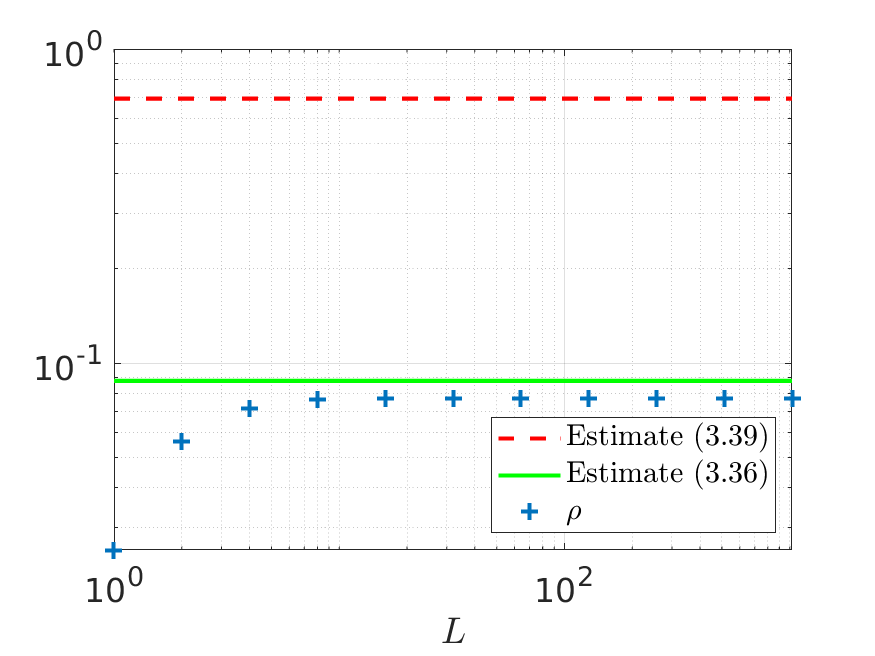}
\caption{Spectral radius of \MG{the} preconditioned matrix as a function of $L$. 
Left: fixed value of $T$ (with $T=1$), Right: $T=L\Delta T$.}\label{scal}
\end{figure}
In both cases, we observe perfect scalability of \MMG{ParaOpt}, in the sense that the spectral radius is uniformly bounded with respect to the number of subintervals considered in the time parallelization.


\subsection{A nonlinear example}\label{lotka_voltera_section}

\MG{We now consider} a control problem \MG{associated} with a nonlinear vectorial dynamics, namely the \emph{Lotka-Volterra} system.
The problem consists \FFK{of} minimizing the cost functional$$ J(c) = \frac{1}{2}|y(T)-y_{{target}}|^2 + \frac\alpha 2 \int_0^{T} |c(t)|^2\,dt$$
with $y_{{target}}=(100,20)^T$, subject to the Lotka-Volterra equation\MG{s}
\begin{equation}\label{eq:lotka_volterra}
\begin{aligned}
\dot{y}_1 &= g(y):=a_1y_1 - b_1y_1y_2+c_1,\\
\dot{y}_2 &= \widetilde{g}(y):= a_2y_1y_2 - b_2y_2+c_2
\end{aligned}
\end{equation}
\JJS{with $a_1=b_2=10$, $b_1=a_2=0.2$ and} initial conditions $y(0) = (20,10)^T$. In this nonlinear setting, the computation of each component of ${\cal F}(Y,\Lambda)$ for given $Y$ and $\Lambda$ requires a series of independent iterative inner loops. In our test, these computations are carried out using a Newton method. As in Section~\ref{Sec3}, the time discretization of~\eqref{BVP} is performed with an implicit Euler scheme.

\JJJS{In a first test, we set $T=\MMG{1/3}$ and $ \alpha=5\FK{\times}
  10^{-2}$ and fix the fine time discretization step \MG{to} $\delta t
  = \frac T {N_0}$, with $N_0=12\cdot 10^{-5}$.}
\MMG{\FK{In Fig\MG{ure}~\ref{VL_cv}, we show}
the rate of convergence of \MMG{ParaOpt} for \JJS{$L=10$} and various values of the ratio $r= \frac{\delta t}{\Delta t}$. \FFK{Here, the error is defined
as the maximum difference between the interface state and adjoint values obtained from a converged fine-grid solution, and the interface values obtained at each inexact Newton iteration by \MMG{ParaOpt}.}}
\begin{figure}
\centering
\includegraphics[width=.7\linewidth]{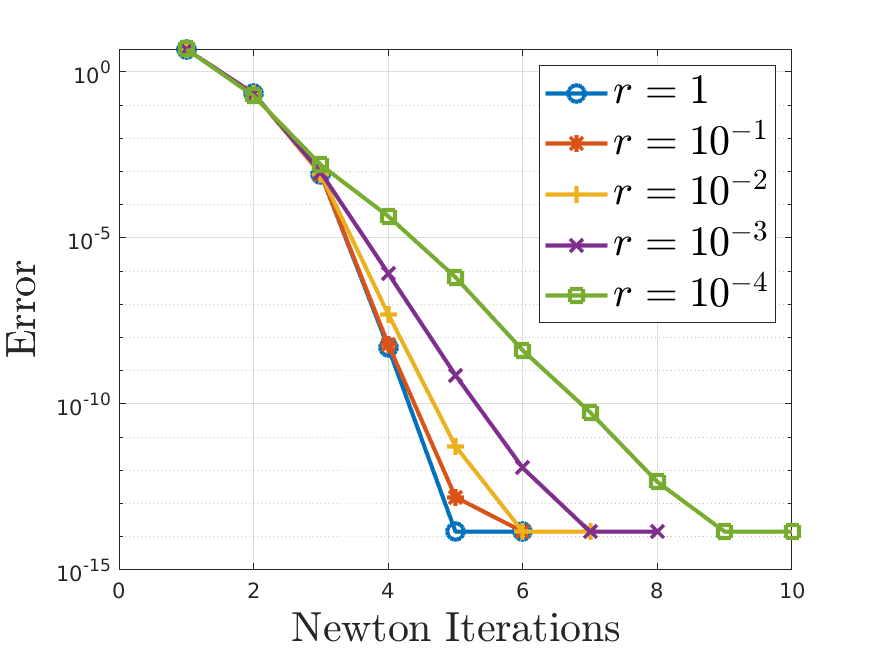}
\caption{\FFK{$L^\infty$ error as a function of the number of exact ($r=1$) or inexact ($r < 1$) Newton iterations}, for various values of the ratio $r= \frac{\delta t}{\Delta t}$.}\label{VL_cv}
\end{figure}
\MG{As can be expected when using a Newton method, we observe that quadratic convergence is obtained in the case $r=1$. \MG{When $r$ becomes smaller,}
the \FK{preconditioning} becomes \FK{a coarser approximation of the exact Jacobian}, \MG{and thus convergence \MMG{becomes a bit slower}}.
}

In our experiments, we \MMG{observed} that the initial \FFK{guess} 
play\FFK{s} a significant role in the convergence of the method. This follows from the fact that \MMG{ParaOpt} \MG{is} an exact (if $\Delta t=\delta t$) or approximate (otherwise) Newton method. The initial 
\FFK{guess} we consider \FFK{is} $c(t)=1$, $y(T_\ell)= (1-T_\ell/T) y_0 +T_\ell/Ty_{target}$, \FFK{and
$\lambda(T_\ell) = (1,1)^T$.}
\MMG{While for $T=1/3$ we observe convergence for all $L$, if we
  increase $T$ to $T=1$,} we do not observe convergence \MMG{any more}
for $L<10$; \FFK{in fact, without decomposing the time domain, the
  sequential version of our solver with $L=1$ does not converge,
  even if we use the exact Jacobian without the coarse
  approximation. This shows that using a}
time-domain decomposition actually helps in solving the nonlinear
problem\MG{, a phenomenon already} observed for a different
time-parallelization method in~\cite{RSSG}.
\FK{These} convergence \FK{problems we observed} are also related to the existence of multiple solutions.
\FK{Indeed, if we coarsen} the outer iteration by replacing the Newton iteration with a Gauss-Newton iteration, i.e., \FK{by} \MG{removing} the second order derivatives of $g$ and $\widetilde{g}$ in Newton's iterative formula, \FK{we obtain} another solution, as illustrated in Fig\MG{ure}~\ref{Limits} \JJJS{on the left for $T=1$ and $r=1$}. \MMG{For both solutions}, we observe that the eigenvalues associated with the linearized dynamics
$$
\delta \dot{y}_1 = a_1\delta y_1 - b_1\delta y_1y_2- b_1y_1\delta y_2+\delta c_1,\quad \delta\dot{y}_2 = a_2\delta y_1y_2 + a_2y_1\delta y_2 - b_2\delta y_2+\delta c_2
$$
\MG{in a neighborhood of the local minima}
remain strictly positive along the trajectories, \MG{in contrast} to the situation analyzed in Section~\ref{Sec3}.  Their values are presented in Fig\MG{ure}~\ref{Limits} \MMG{on the right}.
\begin{figure}
\centering
\includegraphics[width=.48\linewidth]{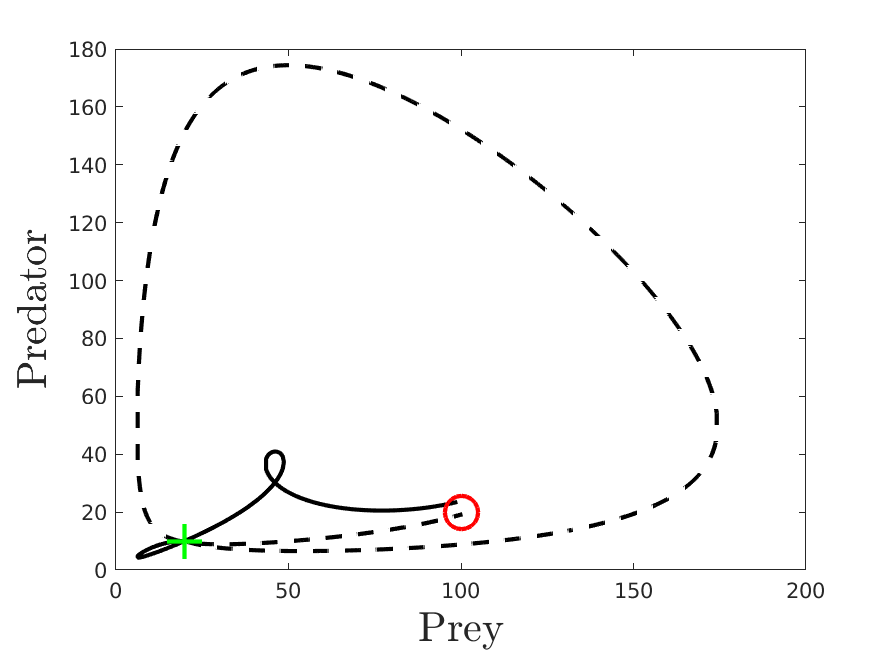}
\includegraphics[width=.48\linewidth]{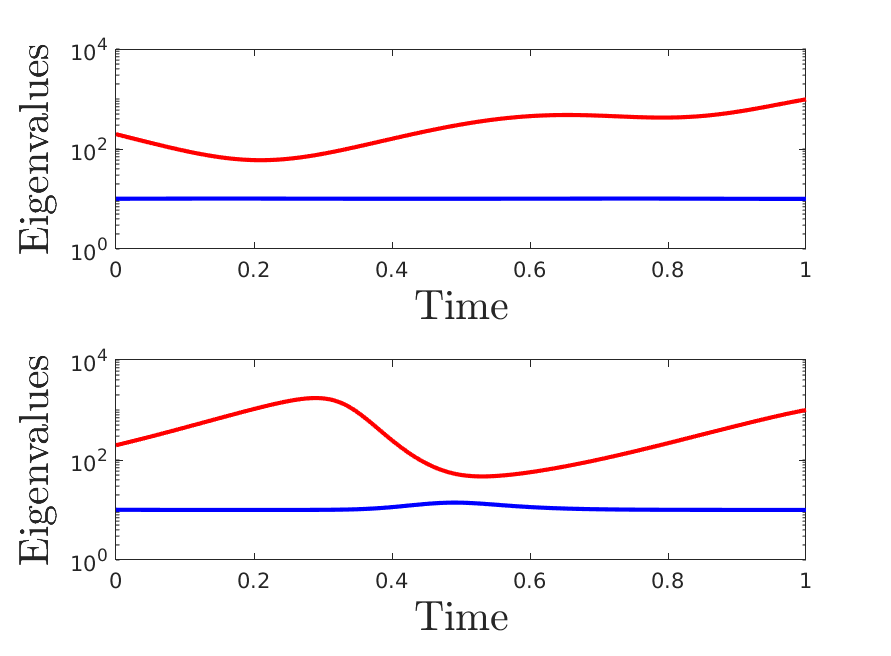}
\caption{Left: two local minima of the cost functional $J$, obtained with Newton (plain line) and Gauss-Newton (dashed line) in the outer loop, for $T=1$. The cost functional values \JJJS{are $J\approx 1064.84$ and $J\approx 15.74$}. The green cross and the red circle indicate $y_0$ and $y_{target}$. Right: \JJJS{(real)} eigenvalues associated with the linearized dynamics in a \FK{neighborhood} of the local minima  obtained with Newton (top) and Gauss-Newton (bottom).}\label{Limits}
\end{figure}


\JJS{We \MMG{next} test the numerical efficiency of our
  algorithm.
\JJJS{The example we consider}
  corresponds to the last curve of Figure~\ref{VL_cv}, i.e. $T=1/3$
  and $r=10^{-4}$, except that we use various values of
  $L\JJJS{\in\{1,3,6,12,24\}}$ \MMG{using the corresponding number of
    processors}.  We execute our code in parallel on workers of
  \MMG{a} parallel pool, using \FFK{{\sc Matlab}'s Parallel Processing
    Toolbox on a 24-core machine that is part of the SciBlade cluster
    at Hong Kong Baptist University.}  The results are presented in
  Table~\ref{tableLV}, where we also indicate the total parallel
  comput\MMG{ing} time without communication,
  \FFK{as well as the number of outer Newton iterations required for
    convergence to a tolerance of $10^{-13}$.}}
\begin{table}	
\centering
\caption{\JJJS{Performance of ParaOpt: total computing time
    $T_{cpu}$, parallel computing time only in seconds and speedup ($T_{cpu}(L=1)/T_{cpu}(L)$)}.
\label{tableLV}}
{\small
\begin{tabular}{|c||c|c|c|c|}
\hline
$L$ & Newton Its. & $T_{cpu}$ & Parallel computing time & speedup  \\
\hline
1 & 14 & 777.53 & 777.42&  1.00 \\
3 & 10 & 172.13 & 167.36&  4.52 \\
6 &  9 &  82.10 & 79.67 &  9.47 \\
12 & 9 &  43.31 & 42.49 & 17.95 \\
24 & 9 &  25.75 & 24.74 & 30.20 \\
\hline
\end{tabular}
}
\end{table}
\JJS{We observe that our cluster enables us to get \MMG{very good scalability, the total computing time is roughly divided by two when the number of processors is doubled.}}
\JJJS{Though not reported in the table, we have observed that even in the case $L=1$, i.e., without parallelization, ParaOpt outperforms the Newton method ($777.53 \ s$ vs. $865.76\ s$ in our test).}

\FFK{To see how this compares with speedup ratios that can be expected from more classical
approaches, we run parareal on the initial value problem \eqref{eq:lotka_volterra} with the same
initial conditions and no control, i.e., $c_1=c_2=0$. For $L=3, 6, 12$ and 24 sub-intervals and
a tolerance of $10^{-13}$, parareal requires $K=3, 6, 8$ and 13 iterations to converge. (For a more
generous tolerance of $10^{-8}$, parareal requires $K=3, 6, 6$ and 7 iterations.) Since the speedup
obtained by parareal cannot exceed $L/K$, the maximum speedup that can be obtained if parareal is used as a subroutine for forward and backward sweeps does not exceed 4 for our problem. 
}\JJJS{Note that this result is specific to the non-diffusive character of the considered equation. This speedup  
would change if the constraint type changed to parabolic, see~\cite[chap. 5]{Nielsen}.}
\subsection{A PDE example}
We finally consider \FK{a control problem involving the} heat equation. More precisely,  Eq.~\eqref{eq:dynamic} is replaced by
$$
\partial_t y - \Delta y=  B c,
$$
where the unknown $y=y(x,t)$ is defined on $\Omega=[0,1]$ with
periodic boundary conditions\FK{, and on}  $[0,T]$ with $T=10^{-2}$. \JJS{Initial and target states are
\begin{align*}
y_{init}	=& \exp(-100(x-1/2)^2),\\
y_{target} =& \frac 12\exp(-100(x-1/4)^2) + \frac 12\exp(-100(x-3/4)^2).
\end{align*}
}
The operator $B$ is the indicator function of a sub-interval $\Omega_c$ of
$\Omega$; in our case, $\Omega_c=[1/3,2/3]$. We \FK{also} set
$\alpha=10^{-4}$. The corresponding solution is shown in
Fig\MG{ure}~\ref{heatcon}.
\begin{figure}
\centering
\includegraphics[width=.7\linewidth]{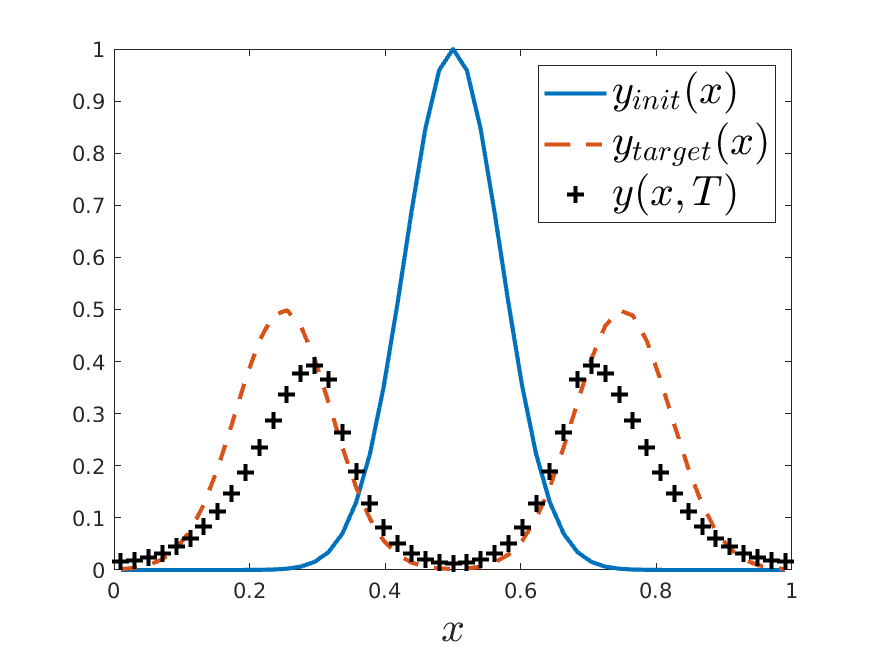}
\caption{Example of initial condition, target state and final state of the solution of the control problem.}\label{heatcon}
\end{figure}
\MG{We use a finite difference scheme \FK{with $50$ grid points} for the \FK{spatial} discretization}.
As in the \MG{previous sub}section, an implicit Euler
scheme is used for the time discretization, and we consider a
parallelization involving $L=10$ subintervals, with $\delta t=10^{-7}$ and $\delta t=10^{-9}$ \JJJS{so that the rate of convergence of the method can be tested for various values of $r=\frac{\delta t}{\Delta t}$.}  
 For $\alpha=10^{-4}$, the evolution of the error along the iterations is shown \MG{in} Fig\MG{ure}~\ref{CVheat}. \FFK{Here, the error is defined as the maximum difference between
 the iterates and the reference discrete solution, evaluated at sub-interval interfaces.}
\begin{figure}
\centering
\includegraphics[width=.46\linewidth]{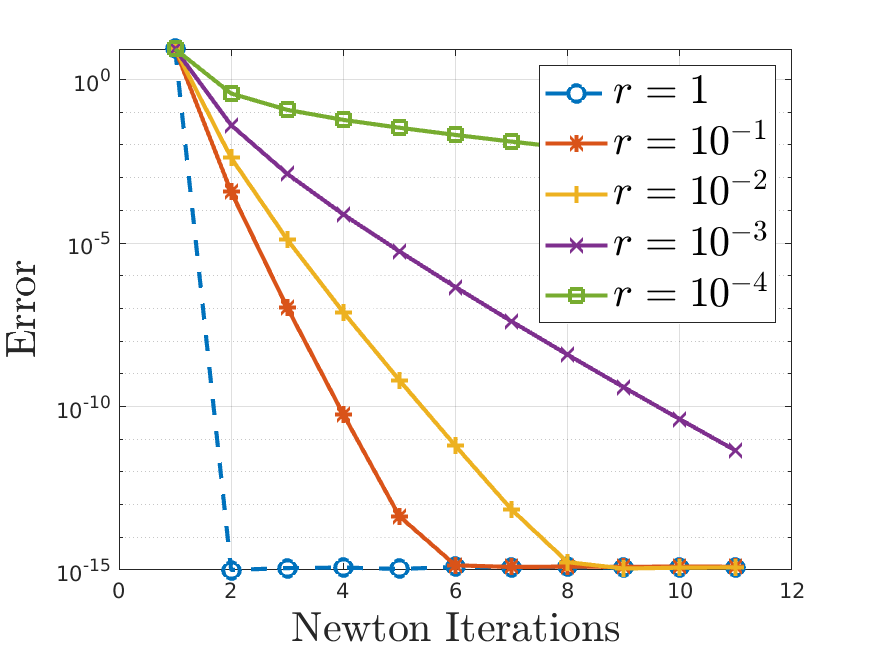}
\includegraphics[width=.46\linewidth]{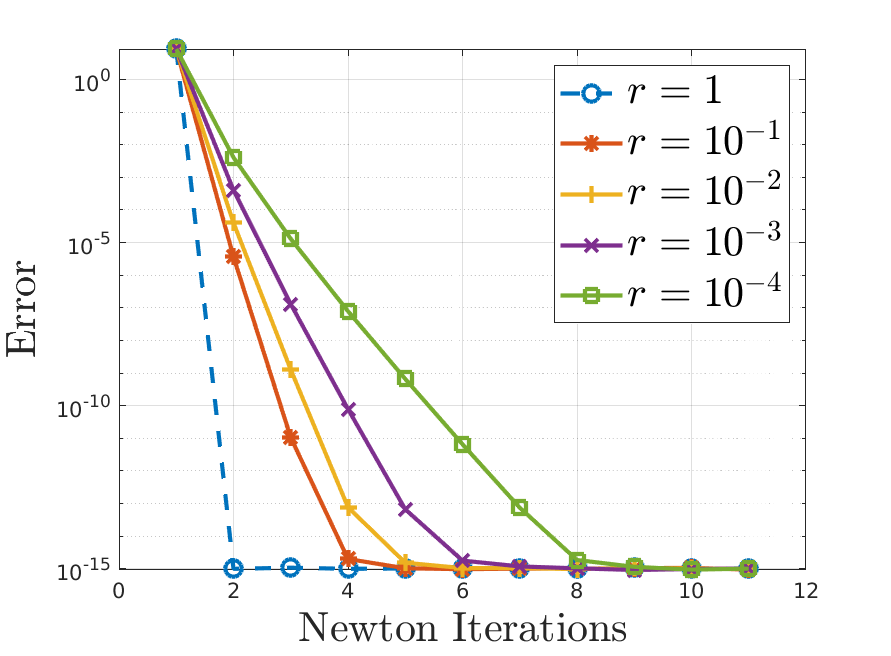}
\caption{\JJJS{Convergence of the method for various values of the ratio $r= \frac{\delta t}{\Delta t}$. Left:  $\delta t=10^{-7}$, right: $\delta t=10^{-9}$. 
}}\label{CVheat}
\end{figure}
\JJJS{Observe also that the convergence curves corresponding to $r=10^{-1}$ and $r=10^{-2}$ on the left panel look nearly identical to the curves for $r=10^{-3}$ and $r=10^{-4}$ on the right panel. This is because they correspond to the same values of $\Delta t$, namely $\Delta t = 10^{-6}$ and $\Delta t = 10^{-5}$. This behavior is consistent with Theorem~\ref{lin_conv_thm}, where the convergence estimate depends only on $\Delta t$, rather than on the ratio $\frac{\delta t}{\Delta t}$.}
\JS{Cases of divergence can also be observed}, in particular for $T=1$ and small values of $\alpha$ and $r$, as shown in Fig\MG{ure}~\ref{Alphaheat}.

\FFK{Of course, one can envisage using different spatial discretizations for the coarse and fine propagators; this may provide additional speedup, provided suitable restriction and prologation operators are used to communicate between the two discretizations. This will be the subject of investigation in a future paper. }
\begin{figure}
\centering
\includegraphics[width=.48\linewidth]{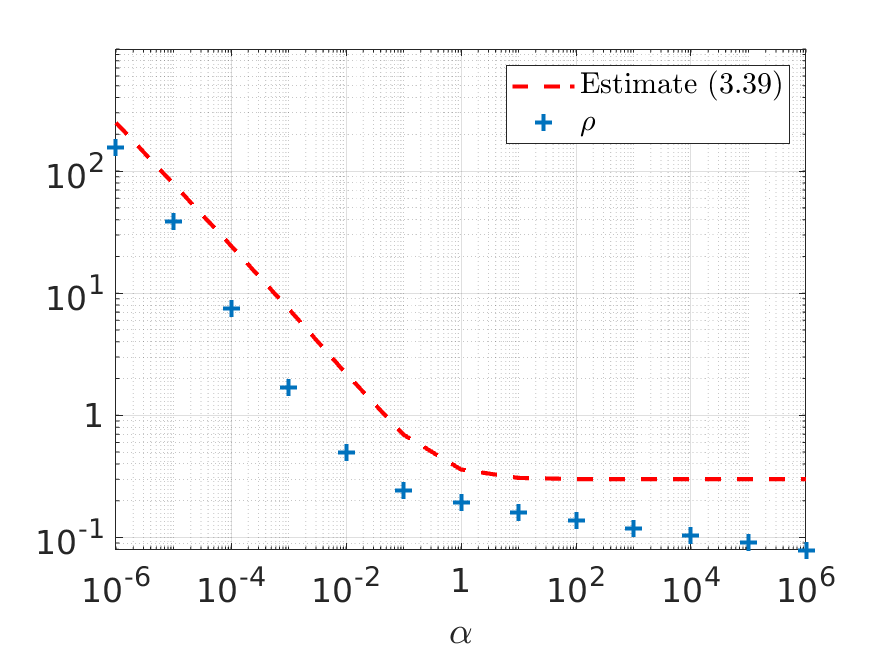}
\includegraphics[width=.48\linewidth]{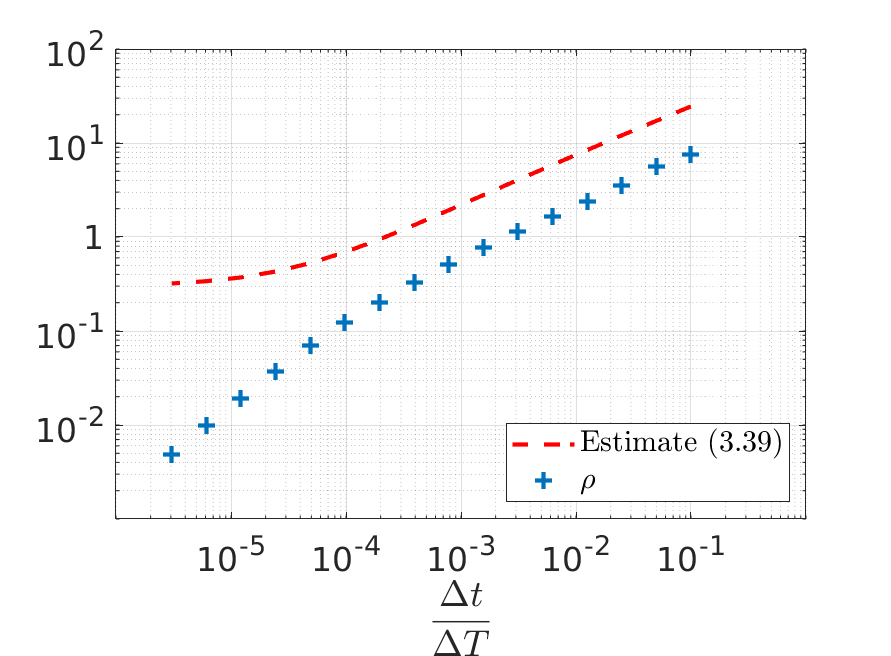}\\
\includegraphics[width=.48\linewidth]{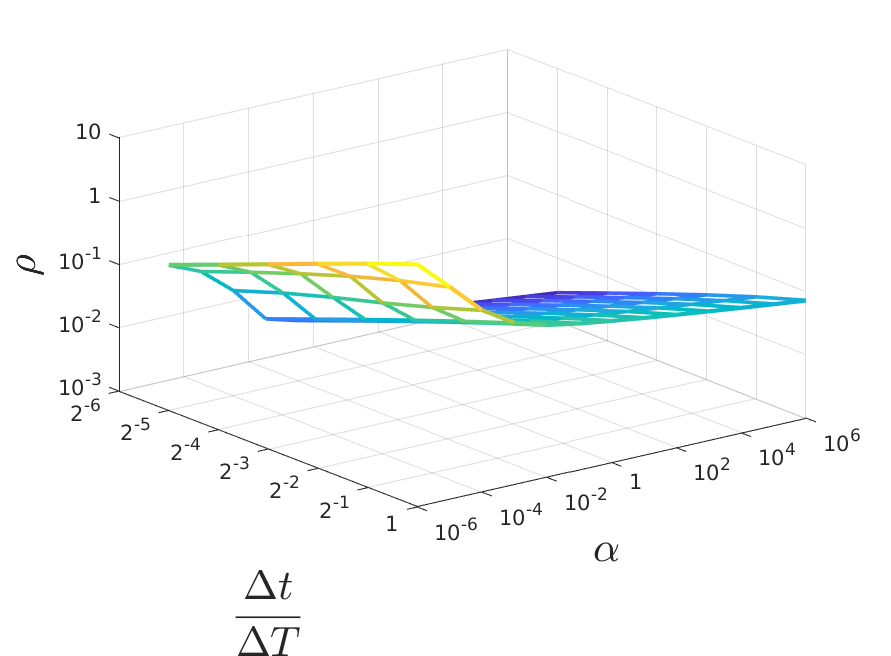}
\includegraphics[width=.48\linewidth]{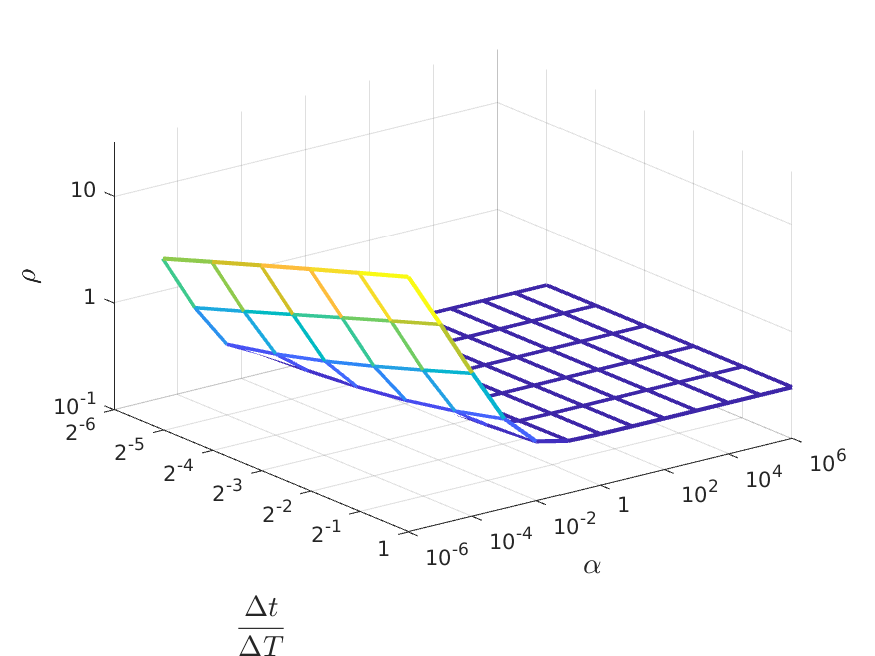}
\caption{Top left: Spectral radius of \MG{the} preconditioned matrix as a function of $\alpha$, with $\delta t=10^{-5}$ and $\Delta t=\Delta T=10^{-1}$.
Top right: Spectral radius of \MG{the}  preconditioned matrix as a function of $\Delta t/\Delta T$, with $\delta t=10^{-8}$ and $\alpha=10^{-4}$.
Bottom left: Spectral radius of \MG{the} preconditioned matrix as a function of $\alpha$ and $\Delta t/\Delta T$, with $\delta t=10^{-7}$.
Bottom right: Estimate~\eqref{eq:spec_radius} as a function of $\alpha$ and $\Delta t/\Delta T$.
}\label{Alphaheat}
\end{figure}

\section{Conclusions}\label{Sec5}

\MG{We introduced a new time\FK{-}parallel algorithm we call ParaOpt
  for time\FK{-}dependent optimal control problems.  \FK{Instead of
    applying Parareal to solve separately the forward and backward
    equations as they appear in an optimization loop, we propose in
    ParaOpt to partition the coupled forward-backward problem directly
    in time, and to use a Parareal-like iteration to incorporate a
    coarse correction when solving this coupled problem.}  We analyzed
  the convergence properties of ParaOpt, \MMG{and} proved \FFK{in the
    linear diffusive case that its convergence is independent of the
    number of sub-intervals in time, \MMG{and thus scalable}. We also}
  tested \MMG{ParaOpt} on \FK{scalar} linear optimal control problems,
  a nonlinear \JJJS{non-diffusive} optimal control problem involving the Lotka-Volterra
  system, and also on a control problem governed by the heat equation.
  \MMG{A small scale parallel implementation of the Lotka-Volterra
    case also showed scalability of ParaOpt for this nonlinear problem.}

  \FFK{Our ongoing work consists of analyzing the algorithm for non-diffusive
  problems. Also, for problems with large state spaces, e.g., for discretized
  PDEs in three spatial dimensions, the approximate Jacobian $\mathcal{J}^G$ in 
  \eqref{approxJacobian} may become too large to solve by direct methods. 
  \JJJS{Thus, we are currently} working on designing efficient preconditioners for solving such systems
  iteratively. Finally, we are currently studying ParaOpt by applying it to realistic problems 
  from applications, in order to better understand its behaviour in such complex cases. }

}

\section*{Acknowledgments}
The authors acknowledge support from ANR Cin\'e-Para (ANR-15-CE23-0019), ANR/RGC ALLOWAPP (ANR-19-CE46-0013/A-HKBU203/19), \FFK{the Swiss National Science Foundation grant no.~200020\_178752,} and the Hong Kong Research Grants Council (ECS 22300115 and GRF 12301817). \FFK{We also thank the anonymous
referees for their valuable suggestions, which greatly improved our paper.}

\appendix

\section{Proof of Inequality \eqref{technical_estimate}}\label{appendix}
Our goal is to prove the following lemma, which is needed for the
proof of Lemma \ref{constant_C}.
\begin{lemma} For every $k>0$ and $0 < x \leq k$, we have \label{lemA.1}
\begin{equation}\label{app:exp}
 (1+x)^{k/x} > k\left(\frac{2+x}{1+x}\right)-1.
\end{equation}
\end{lemma}
First, we need the following property of logarithmic functions.
\begin{lemma} \label{log} For any $x > 0$, we have
$$ \ln(1+x) \geq \frac{x}{x+1} + \frac12\left(\frac{x}{x+1}\right)^2. $$
\end{lemma}
\begin{proof}
Let $u = \frac{x}{x+1} < 1$. Then
\begin{align*}
 \ln(1+x) &= -\ln\left(\frac{1}{1+x}\right) = -\ln(1-u)\\
 &= u + \frac{u^2}{2} + \frac{u^3}{3} + \cdots \geq u + \frac{u^2}{2}.
\end{align*}
The conclusion now follows.
\end{proof}
\begin{proof}(Lemma \ref{lemA.1})
Let $g$ and $h$ denote the left and right hand sides of \eqref{app:exp} respectively. We consider two cases,
namely when $0 < k \leq 1$ and when $k > 1$. When $k \leq 1$, we have
$$ h(x) \leq \frac{2+x}{1+x}-1 = \frac{1}{1+x} < 1 < (1+x)^{k/x} = g(x). $$
For the case $k > 1$, we will show that $g(k) > h(k)$ and
$g'(x)-h'(x) < 0$ for $0 < x < k$, which together imply that $g(x)>h(x)$ for all $0<x\leq k$. The first assertion follows from the fact that
$$ g(k)-h(k) = 1+k - k\cdot \frac{k+2}{k+1} + 1 = 2- \frac{k}{k+1} > 0. $$
To prove the second part, we note that
\begin{align*}
g'(x) &= (1+x)^{k/x}\left[-\frac{k}{x^2}\ln(1+x) + \frac{k}{x(1+x)}\right]\\
&= \frac{-k}{x^2}(1+x)^{k/x - 1}\left[(1+x)\ln(1+x)-x\right] \\
&< \frac{-k}{x^2}(1+x)^{k/x - 1}\cdot\frac{x^2}{2(x+1)} = \frac{-k}{2}(1+x)^{k/x - 2} < 0,\\
h'(x) &= -\frac{k}{(1+x)^2} < 0.
\end{align*}
Therefore, we have
$$ g'(x)-h'(x)
<  -\frac{k}{(1+x)^2}\left[\frac12(1+x)^{k/x}-1\right]\\
\leq  -\frac{k}{(1+x)^2}\underbrace{\left[\frac{1+k}{2}-1\right]}_{\text{$>0$ since $k>1$}} < 0.
$$
Thus, $g(x) > h(x)$ for all $0<x<k$, as required.
\end{proof}


\bibliographystyle{siam}
\bibliography{BiblioGKS}
\end{document}